\RequirePackage[l2tabu, orthodox]{nag}
\documentclass[a4paper,11pt]{amsart}
\usepackage[margin=32mm,bottom=40mm]{geometry}

\usepackage{lmodern}
\usepackage[stretch=10]{microtype}
\usepackage[T1]{fontenc}
\usepackage{tikz, tikz-cd, stmaryrd, amsmath, amsthm, amssymb,
hyperref, comment, todonotes, bbm}
\usepackage[shortlabels]{enumitem}
\usetikzlibrary{arrows}
\usetikzlibrary{positioning}
\usepackage[utf8x]{inputenc}
\newcommand{\bb}{\textbf}

\newcommand{\ol}{\overline}
\newcommand{\mc}{\mathcal}
\newcommand{\wh}{\widehat}
\newcommand{\wt}{\widetilde}
\newcommand{\mf}{\mathfrak}
\renewcommand{\AA}{\mathbb{A}}

\newcommand{\HH}{\mathbb{H}}
\newcommand{\ZZ}{\mathbb{Z}}

\newcommand{\PP}{\mathbb{P}}

\newcommand{\FF}{\mathbb{F}}

\DeclareMathOperator{\Ind}{Ind}
\DeclareMathOperator{\pr}{pr}
\DeclareMathOperator{\tr}{tr}

\DeclareMathOperator{\diag}{diag}

\DeclareMathOperator{\Divv}{Div}

\DeclareMathOperator{\coker}{coker}

\DeclareMathOperator{\Span}{Span}
\DeclareMathOperator{\res}{res}

\DeclareMathOperator{\ord}{ord}
\DeclareMathOperator{\Spec}{Spec}

\DeclareMathOperator{\cha}{char}

\DeclareMathOperator{\qlog}{qlog}

\theoremstyle{plain}
\newtheorem{Theorem}{Theorem}[section]

\newtheorem{Lemma}[Theorem]{Lemma}
\newtheorem{Corollary}[Theorem]{Corollary}

\newtheorem{Proposition}[Theorem]{Proposition}

\newtheorem{Question}[Theorem]{Question}

\theoremstyle{definition}

\theoremstyle{remark}
\newtheorem{Remark}[Theorem]{Remark}

\numberwithin{equation}{section}
\begin{document}

\title{$p$-group Galois covers of curves in characteristic $p$}
\author[J. Garnek]{J\k{e}drzej Garnek}
\address{\parbox{\linewidth}{Faculty of Mathematics and Computer Science,
		Adam Mickiewicz University,\\
		ul. Uniwersytetu Pozna\'{n}skiego~4, \mbox{61-614} Poznan, Poland \\ \mbox{}}}
\email{jgarnek@amu.edu.pl}
\subjclass[2020]{Primary 14F40, Secondary 14G17, 14H30}
\keywords{de Rham cohomology, algebraic curves, group actions,
	characteristic~$p$}
\urladdr{http://jgarnek.faculty.wmi.amu.edu.pl/}
\date{}  

\begin{abstract}
	We study cohomologies of a curve with an action of a finite $p$-group over a field of characteristic $p$.
	Assuming the existence of a certain ``magical element'' in the function field of the curve, we compute the equivariant structure of the module of holomorphic differentials and the de Rham cohomology, up to certain local terms. We show 
	that a generic $p$-group cover has a ``magical element''.
	As an application we compute the de Rham cohomology
	of a curve with an action of a finite cyclic group of prime order.
\end{abstract}

\maketitle
\bibliographystyle{plain}
\vspace{-3em}
\section{Introduction} \label{sec:intro}

Let $X$ be an algebraic curve with an action of a finite group $G$ over an algebraically closed field $k$.
Studying the $k[G]$-module structure of cohomologies of $X$ is a natural and fundamental topic in algebraic geometry.
In the classical  case,  that is, over the field of complex numbers, the equivariant structure
of the module of holomorphic differentials was completely determined  by  Chevalley and Weil, cf. \cite{Chevalley_Weil_Uber_verhalten}. Their result remains valid when $\cha k \nmid \# G$ (see e.g. \cite{Ellingsrud_Lonsted_Equivariant_Lefschetz}). \\

When $\cha k = p > 0$ and $p| \# G$, the structure of $H^0(X, \Omega_X)$ becomes much more complicated. Most of previous results on this subject concern either the case when the
ramification is under control (cf. e.g. \cite{Kani_Galois_module}, \cite{Nakajima_Galois_module}, \cite{chinburg_epsilon_constants} for the case of tame ramification and~\cite{Kock_galois_structure} for the case of a weakly ramified action) or focus on some special
groups (see e.g.~\cite{Valentini_Madan_Automorphisms} for the case of cyclic groups, \cite{WardMarques_HoloDiffs} for abelian groups or~\cite{Bleher_Chinburg_Kontogeorgis_Galois_structure} for groups with a cyclic Sylow subgroup).
Other results compute the structure of the space of global sections of
an invertible sheaf of sufficiently large degree (cf. \cite{Borne_Cohomology_of_G_sheaves}).
Even less is known about the equivariant structure of the de Rham cohomology.
In the tame ramification case the de Rham--Euler characteristic class has been computed in the $K$-theory category (cf.~\cite{Chinburg1994} and~\cite{chinburg_epsilon_constants}).
There are also results concerning the de Rham cohomology of Deligne--Lusztig curves (cf. \cite{Lusztig_Coxeter_orbits}), of Hermitian curves (cf. \cite{Dummigan_95}, \cite{Dummigan_99} and \cite{Haastert_Jantzen_Filtrations})
and of Suzuki curves (cf. \cite{Liu_decomposition_suz}, \cite[p. 2535]{Gross_Rigid_local_systems_Gm} and \cite{Malmskog_Pries_Weir_dR_Suzuki}). In general, it is known that the Hodge--de Rham exact sequence does not split (see \cite{Hortsch_canonical_representation}, \cite{KockTait2018} for explicit examples and \cite[Main Theorem]{Garnek_equivariant} for my result that yields a more precise criterion).
One should also mention, that there are several results
concerning the structure of $H^1_{dR}(X)$ as a Dieudonn\'{e} module for a curve~$X$ with a $\ZZ/p$-action, cf. \cite{PriesZhu_p_rank_AS} and~\cite{Booher_Cais_a_numbers}.\\

In this paper we would like to propose a new point of view on this problem.
Keep the above notation. Recall that the Hodge--de Rham spectral sequence of~$X$ degenerates on the first page,
leading to the exact sequence:
\begin{equation} \label{eqn:intro_hodge_de_rham_se}
	0 \to H^0(X, \Omega_X) \to H^1_{dR}(X) \to H^1(X, \mc O_X) \to 0,
\end{equation}
which we call the Hodge--de Rham exact sequence. Our previous article \cite{Garnek_equivariant} concerned
the equivariant behaviour of the exact sequence \eqref{eqn:intro_hodge_de_rham_se}. In particular, we have proven that
under some mild assumptions the difference of the dimensions of $G$-invariant subspaces 
is a sum of local terms indexed by the points of the curve $X$:
\begin{equation} \label{eqn:decomposition_of_defect}
	\dim_k H^0(X, \Omega_X)^G + \dim_k H^1(X, \mc O_X)^G - \dim_k H^1_{dR}(X)^G = \sum_{P \in X(k)} \delta_P
\end{equation}
(cf. \cite[Proposition 3.1]{Garnek_equivariant}). The local term $\delta_P$ depends only on the completed local ring $\wh{\mc O}_{X, P}$ and measures the ramification of the action at the given point $P \in X(k)$. For example,
one can prove that $\delta_P = 0$ if and only if 
$G_{P, 2}$, the second ramification group at $P$, vanishes (cf. \cite[proof of Main Theorem]{Garnek_equivariant}).

The formula~\eqref{eqn:decomposition_of_defect} is surprising, as the left hand side is defined globally
and the right hand side is a sum of local terms. We conjecture that the reason behind this phenomenon is that
the Hodge cohomology $H^1_{Hdg}(X) := H^0(X, \Omega_X) \oplus H^1(X, \mc O_X)$ and the de Rham cohomology $H^1_{dR}(X)$
decompose as $k[G]$-modules into certain global and local parts:
\begin{align}
	H^1_{Hdg}(X) &\cong \textrm{(global part)} \oplus \bigoplus_{P \in X(k)} H^1_{Hdg, P}, \label{eqn:decomposition_Hdg}\\
	H^1_{dR}(X) &\cong \textrm{(global part)} \oplus \bigoplus_{P \in X(k)} H^1_{dR, P}. \label{eqn:decomposition_dR}
\end{align}
More precisely, the global part should depend only on the ``topology'' of the cover $\pi : X \to X/G$ (i.e. on the curve $Y$ and the stabilizer subgroups) and be the same for both cohomologies.
The local parts $H^1_{Hdg, P}$, $H^1_{dR, P}$ should depend only on the completed local rings $\wh{\mc O}_{X, P}$ of a given point~$P$. The decompositions~\eqref{eqn:decomposition_Hdg} and~\eqref{eqn:decomposition_dR} would yield a formula of the form~\eqref{eqn:decomposition_of_defect} with:
\[
\delta_P := \dim_k (H^1_{Hdg, P})^G - \dim_k (H^1_{dR, P})^G.
\]
The importance of decompositions~\eqref{eqn:decomposition_Hdg} and~\eqref{eqn:decomposition_dR} is that they would allow computation of the equivariant structure of the cohomologies using some local considerations.
The goal of this article is to prove a weaker form of the decompositions~\eqref{eqn:decomposition_Hdg} and~\eqref{eqn:decomposition_dR} for a generic $p$-group cover.\\

Throughout this paper, we work in the following setup (unless specified otherwise).
Suppose that $G$ is a finite $p$-group and $k$ is an algebraically closed field
of characteristic~$p$. Let $X$ be as above and denote by $\pi : X \to Y := X/G$ the canonical morphism. Let $g_Y$ be the genus of $Y$ and $B \subset Y(k)$ -- the branch locus of $\pi$. For any ${P \in X(k)}$ denote by $G_{P, i}$ the $i$-th ramification group
at $P$ and let:
\[
d_{P}:= \sum_{i \ge 0} (\# G_{P, i} - 1), \quad 
d_{P}' := \sum_{i \ge 1} (\# G_{P, i} - 1), \quad
d_{P}'' := \sum_{i \ge 2} (\# G_{P, i} - 1)
\]
(note that $d_P$ is the exponent of the different of $k(X)/k(Y)$ at $P$). We assume that the cover $\pi$ satisfies the following assumptions:
\begin{enumerate}[(A)]
	\item \label{enum:A} the stabilizer $G_P$ of $P$ in $G$ is a normal subgroup of $G$ for every $P \in X(k)$,
	
	\item \label{enum:B} there exists a function $z \in k(X)$ (a ``magical element'') satisfying $\ord_P(z) \ge -d_P'$
	for every $P \in X(k)$ and $\tr_{X/Y}(z) \neq 0$.
\end{enumerate}
We discuss~\ref{enum:A} and~\ref{enum:B} below, in Subsection \emph{Constructing magical elements}. 
By the assumption~\ref{enum:A}, for $Q \in Y(k)$ we may denote $G_Q := G_P$ and $d_Q := d_P$
for any $P \in \pi^{-1}(Q)$. Let $I_G := \{ \sum_{g \in G} a_g g \in k[G] : \sum_{g \in G} a_g = 0 \}$ be the augmentation ideal of the group~$G$
and let $J_G := I_G^{\vee}$ be the dual of $I_G$. For any subgroup $H \le G$
we consider also the relative augmentation ideal $I_{G, H} := \Ind^G_H \, I_H$
and its dual $J_{G, H} := \Ind^G_H \, J_H$. Finally, the $k[G]$-modules
$I_{X/Y}$ and $J_{X/Y}$ are defined by:
\begin{align*}
	I_{X/Y} &:= \ker \left( \sum : \bigoplus_{Q \in Y(k)} I_{G, G_Q} \to I_G \right),\\
	J_{X/Y} &:= \coker \left( \diag : J_G \to  \bigoplus_{Q \in Y(k)} J_{G, G_Q} \right).
\end{align*}
The following is the main result of this paper. It is proven in Sections~\ref{sec:OmegaX} and~\ref{sec:dR}.
\begin{Theorem} \label{thm:main_thm} 
	Keep the above assumptions. We have the following isomorphisms of $k[G]$-modules:
	\begin{align*}
		H^0(X, \Omega_X) &\cong k[G]^{\oplus g_Y} \oplus I_{X/Y} \oplus \bigoplus_{Q \in B} H^0_Q,\\
		H^1(X, \mc O_X) &\cong k[G]^{\oplus g_Y} \oplus J_{X/Y} \oplus \bigoplus_{Q \in B} H^1_Q,\\
		H^1_{dR}(X) &\cong k[G]^{\oplus 2 \cdot g_Y} \oplus I_{X/Y} \oplus J_{X/Y} 
		\oplus \bigoplus_{Q \in B} H^1_{dR, Q},
	\end{align*}
	where $H^0_Q$, $H^1_Q$, $H^1_{dR, Q}$ are certain $k[G]$-modules that depend only on the rings $\mc O_{X, Q}$ and on the element $z$ (see~\eqref{eqn:H0Q}, \eqref{eqn:H1Q} and~\eqref{eqn:H1dRQ}
	for precise definitions).
\end{Theorem}
We explain now the relation between Theorem~\ref{thm:main_thm} and the conjectural decompositions~\eqref{eqn:decomposition_Hdg}, \eqref{eqn:decomposition_dR}.
The terms $k[G]^{\oplus g_Y} \oplus I_{X/Y}$, $k[G]^{\oplus g_Y} \oplus J_{X/Y}$, $k[G]^{\oplus 2g_Y} \oplus I_{X/Y} \oplus J_{X/Y}$ play role of global parts, as they
depend only on $g_Y$, $\# B$ and $\{ G_Q : Q \in B \}$. The candidates for local parts are the modules $H^0_Q$, $H^1_Q$, $H^1_{dR, Q}$. However, they depend on the element~$z$ from the condition~\ref{enum:B}. Conjecturally, one should be able to describe them only in terms of $\wh{\mc O}_{X, Q}$. We expect also that $\dim_k H^0_Q = \frac{1}{2} \# \pi^{-1}(Q) \cdot d_Q''$ for every $Q \in B$ (cf. Lemma~\ref{lem:properties_H0Q_H1Q}~(3) for a motivation).\\

As an application of Theorem~\ref{thm:main_thm}, we give a description of cohomologies of $\ZZ/p$-covers (cf. Section~\ref{sec:Zp}). 
Let $J_i$ denote the unique indecomposable $k[\ZZ/p]$-module of dimension~$i$ for $i = 1, \ldots, p$.
\begin{Corollary} \label{cor:cohomology_of_Zp}
	Keep the above assumptions with $G = \ZZ/p$.
	Suppose that $\pi$ has a global standard form (cf. Subsection~\ref{subsec:gsf}). Then, as $k[\ZZ/p]$-modules:
	\begin{align*}
		H^0(X, \Omega_X) &\cong H^1(X, \mc O_X) \cong J_{p}^{\oplus g_Y} \oplus J_{p-1}^{\oplus (\# B - 1)} \oplus
		\bigoplus_{Q \in B} \bigoplus_{i = 1}^{p-1} J_i^{\oplus \alpha_Q(i)}, \\
		H^1_{dR}(X) &\cong J_{p}^{\oplus 2 \cdot g_Y} \oplus J_{p-1}^{\oplus \alpha},
	\end{align*}
	where $g_Y$ is the genus of $Y$, $B \subset Y(k)$ denotes the branch locus of $\pi$, $m_Q := d_Q'/(p-1)$ and:
	\begin{align*}
		\alpha_Q(i) &:= \left \lceil \frac{m_Q \cdot (i+1)}{p} \right \rceil
		- \left \lceil \frac{m_Q \cdot i}{p} \right \rceil,\\
		\alpha &:= \sum_{Q \in B} (m_Q + 1) - 2.
	\end{align*}
\end{Corollary}
For $H^0(X, \Omega_X)$ this result was already known (cf.~\cite[Theorem~1]{Valentini_Madan_Automorphisms}), however we don't know of any previous results regarding the de Rham cohomology.\\

Theorem~\ref{thm:main_thm} allows us to give a converse of~\cite[Main Theorem]{Garnek_equivariant}
for covers satisfying~\ref{enum:A} and~\ref{enum:B}.
\begin{Corollary} \label{cor:hdr_exact_sequence}
	Keep assumptions of Theorem~\ref{thm:main_thm}. Then the action of $G$ on $X$ is weakly ramified
	if and only if the exact sequence~\eqref{eqn:intro_hodge_de_rham_se} splits $G$-equivariantly.
\end{Corollary}
One could hope that a similar decomposition to the one in Theorem~\ref{thm:main_thm}
holds in a more general context, for example for the de Rham cohomology treated as a 
$k[F, V]$-module (where $F$ and $V$ denote the Frobenius and Verschiebung morphisms).
A result of Elkin and Pries accomplishes this in the case when $Y = \PP^1$ and $p = 2$ (cf.~\cite[Theorem~1.2]{elkin_pries_ekedahl_oort}). However, explicit examples show that such a decomposition is impossible in general, cf. \cite[Example~7.2]{Booher_Cais_a_numbers}.\\

Our results leave some questions open. In which generality do the decompositions~\eqref{eqn:decomposition_Hdg}
and~\eqref{eqn:decomposition_dR} hold? How to describe
the modules $H^0_Q$, $H^1_Q$, $H^1_{dR, Q}$ without using the magical element? 
Is it possible to generalize the above considerations to higher dimensional varieties or to crystalline
cohomology? We plan to investigate those questions in the near future.
\subsection*{Strategy of the proof of Theorem~\ref{thm:main_thm}}
We explain now the idea behind the proof of Theorem~\ref{thm:main_thm} for the module of holomorphic differentials. The proof for $H^1(X, \mc O_X)$ and
$H^1_{dR}(X)$ follows the same strategy. Our approach is divided into two main steps.
In both steps we use the fact that the function $z$ is a normal basis
element of $k(X)/k(Y)$.\\

In the first step we compare the module $H^0(X, \Omega_X)$ with $H^0(X, \Omega_X(R))$, the module of differentials with logarithmic poles in the ramification locus of the cover (cf. Proposition~\ref{prop:log_diffs_and_diffs}). To this end we write any form $\omega \in \Omega_{k(X)}$ as $\sum_{g \in G} g^*(z) \omega_g$ for $\omega_g \in \Omega_{k(Y)}$ and
consider the following question.
\begin{Question}
	Suppose that $\omega \in H^0(X, \Omega_X)$. What are the possible values of the residues $\res_Q(\omega_g)$ for $Q \in B$ and $g \in G$? In other words, what is the image of $H^0(X, \Omega_X)$ under the map:
	\[
	\res_G : \Omega_{k(X)} \to \bigoplus_{Q \in B} k[G], \qquad \omega \mapsto \sum_{Q \in B} \sum_{g \in G} \res_Q(\omega_g) g_Q,
	\]
	where $g_Q \in \bigoplus_{B} k[G]$ is the element with $g$ on the $Q$-th component and $0$ on other components?
\end{Question}
\noindent There are two conditions imposed on the residues of $\omega_g$:
\begin{enumerate}[(1)]
	\item The first one follows from the residue theorem: $\sum_{Q \in B} \res_Q(\omega_g) = 0$.
	
	\item The second one is imposed by the condition $\res_P(\omega) = 0$ for every $P \in \pi^{-1}(B)$.
\end{enumerate}
The conditions (1) and (2) define the module $I_{X/Y}$. In this way we obtain an equivariant homomorphism
\begin{equation} \label{eqn:map_to_IXY}
	\res_G : H^0(X, \Omega_X) \to I_{X/Y}.
\end{equation}
It turns out that this homomorphism is split. One constructs
a section by choosing appropriate forms in $\Omega_{k(Y)}$ with known residues.
We can apply a similar reasoning for logarithmic differential forms, neglecting the condition (2).
In this way one obtains a split homomorphism:
\begin{equation} \label{eqn:map_to_kGB}
	\res_G : H^0(X, \Omega_X(R)) \to k[G]_B,
\end{equation}
where $k[G]_B$ is the submodule of $\bigoplus_B k[G]$ defined by the condition (1). It turns out that the maps~\eqref{eqn:map_to_IXY} and~\eqref{eqn:map_to_kGB}
have the same kernel (cf. Proposition~\ref{prop:log_diffs_and_diffs}).\\

In the second step we observe that we have an inclusion of sheaves of the same rank:
\[
\bigoplus_{g \in G} g^*(z) \Omega_Y(B) \subset \Omega_X(R).
\]
Hence, their quotient $\mc T$
is a torsion sheaf and its global sections decompose as a sum of local parts. By
applying the long exact sequence to the sequence of sheaves:
\[
0 \to \bigoplus_{g \in G} g^*(z) \Omega_Y(B) \to \Omega_X(R) \to \mc T \to 0
\]
and observing that $k[G]$ is an injective $k[G]$-module, we find the $k[G]$-structure on $H^0(X, \Omega_X(R))$ (this is done at the beginning of Section~\ref{sec:OmegaX} on page~\pageref{eqn:structure_Omega_X_R}).
\subsection*{Constructing magical elements}
It is reasonable to ask how often are the conditions~\ref{enum:A} and~\ref{enum:B} satisfied.
Unfortunately, not every $p$-group cover has a magical element, cf. Subsection~\ref{subsec:no_magical_element} for a counterexample.
However, it turns out that a generic $G$-cover satisfies~\ref{enum:A} and~\ref{enum:B}. To be precise,
let $k$ be an algebraically closed field of characteristic~$p$. Fix a finite $p$-group~$G$ and an affine open subset $U$ of a smooth projective curve $Y$ over $k$. Let $M_{U, G}$ denote the moduli space
of pointed $G$-covers of $Y$ unramified over $U$, as defined in~\cite{Harbater_moduli_of_p_covers}.

\begin{Theorem} \label{thm:generic_intro}
	The set of covers satisfying~\ref{enum:A} and~\ref{enum:B} forms a dense subset of~$M_{U, G}$.
\end{Theorem}
\noindent The proof of Theorem~\ref{thm:generic_intro} gives a way to construct inductively
covers satisfying~\ref{enum:A} and~\ref{enum:B}.  Subsection~\ref{subsec:example_heisenberg} gives an explicit family of $E(p^3)$-covers $\pi : X \to Y$ satisfying
(A) and (B), where $Y$ is an elliptic curve and $E(p^3)$ denotes the Heisenberg group $\textrm{mod } p$.

We discuss now main ideas behind the proof of Theorem~\ref{thm:generic_intro}.
Suppose that $X \to Y$ factors through a Galois cover $X' \to Y$.
It turns out that if both $X \to X'$ and $X' \to Y$ have magical elements, then
$X \to Y$ also has a magical element (cf. Lemma~\ref{lem:new_magical_elts}). Hence everything comes down to
constructing magical elements for $\ZZ/p$-covers.
We prove that a global standard form of a $\ZZ/p$-cover (cf. Subsection~\ref{subsec:gsf} for a definition)
yields a magical element. Moreover, every sufficiently ramified $\ZZ/p$-cover
has a global standard form (cf. Lemma~\ref{lem:criterion_for_gsf}).
Theorem~\ref{thm:generic_intro} follows by noting that a generic $G$-cover can be factored
into sufficiently ramified $\ZZ/p$-covers.
\subsection*{Outline of the paper}
In Section~\ref{sec:notation} we give necessary notation and preliminaries. Section~\ref{sec:magical_elements}
proves some properties of the magical element $z$ from the condition~\ref{enum:B}.
This allows us to prove the part of Theorem~\ref{thm:main_thm} concerning the module of holomorphic differentials
and the cohomology of the structure sheaf in Section~\ref{sec:OmegaX}. In Section~\ref{sec:OX} we study the dual of the map $\res_G$. We prove
the decomposition of the de Rham cohomology from Theorem~\ref{thm:main_thm} in Section~\ref{sec:dR}.
In Section~\ref{sec:AS_covers} we introduce the notion of a global standard form
of a $\ZZ/p$-cover. This allows us to construct a magical element for a large class
of Artin--Schreier covers. Also, we prove Corollary~\ref{cor:cohomology_of_Zp}.
Finally, in Section~\ref{sec:constructing_magical_elements} we prove that a 
generic $p$-group cover has a magical element and give an example of an $E(p^3)$-cover
satisfying~\ref{enum:A} and~\ref{enum:B}.

\subsection*{Acknowledgements}
The author wishes to express his gratitude to Bartosz Naskręcki and Wojciech Gajda,
whose comments helped to considerably improve the exposition of the paper.
The ``global--local'' point of view on this problem was inspired by a conversation with
Piotr Achinger in November 2018. Finally, the author would like to thank the anonymous reviewer, who noted a lot of errors in the original manuscript and suggested a way to significantly simplify the proof of Theorem~\ref{thm:main_thm}. The author was supported by the grant 038/04/N\'{S}/0011, which is a part of the
project "Initiative of Excellence -- Research University" on Adam Mickiewicz University, Poznan.

\section{Preliminaries} \label{sec:notation}
We start this section by introducing notation concerning algebraic curves used throughout the paper.
For an arbitrary smooth projective curve $Y$ over a field $k$ we denote by $k(Y)$ the function field of $Y$.
Also, we write $\ord_Q(f)$ for the order of vanishing of a function $f \in k(Y)$ at a point $Q \in Y(k)$.
Let $\mf m_{Y, Q}^n := \{ f \in k(Y) : \ord_Q(f) \ge n \}$ for any $n \in \ZZ$. To simplify notation, we write
$\Omega_Y$, $\Omega_{k(Y)}$ and $H^1_{dR}(Y)$ instead of $\Omega_{Y}$, $\Omega_{k(Y)/k}$ and $H^1_{dR}(Y/k)$.
We will often identify a finite set $S \subset Y(k)$ with a reduced divisor in $\Divv(Y)$.
Thus e.g. $\Omega_Y(S)$ will denote the sheaf of logarithmic differential
forms with poles in $S$. In the sequel we often use residues of differential forms,
see e.g. \cite[Remark III.7.14]{Hartshorne1977} for relevant facts.\\

Let $G$ be a finite group and $\pi : X \to Y$ be a finite separable $G$-cover of smooth projective curves over a field $k$.
In the sequel we identify $\Omega_{k(Y)}$ with a submodule of $\Omega_{k(X)}$ and
$k(Y)$ with a subfield of $k(X)$. We denote the ramification index of $\pi$ at $P \in X(k)$ by $e_{X/Y, P}$ and
by $G_{P, i}$ the $i$-th ramification group of $\pi$ at $P$, i.e.
\[
G_{P, i} := \{ \sigma \in G : \sigma(f) \equiv f \mod{\mf m_P^{i+1}} \quad \textrm{ for every } f \in \mc O_{X, P} \}.
\]
Also, we use the following notation:
\begin{align*}
	d_{X/Y, P}:= \sum_{i \ge 0} (\# G_{P, i} - 1), \quad
	d_{X/Y, P}' := \sum_{i \ge 1} (\# G_{P, i} - 1), \quad
	d_{X/Y, P}'' := \sum_{i \ge 2} (\# G_{P, i} - 1)
\end{align*}
($d_{X/Y, P}$ is the exponent of the different of $k(X)/k(Y)$ at $P$, cf. \cite[Proposition~IV.\S 1.4]{Serre1979}).
Recall that for any $P \in X(k)$ and $\omega \in \Omega_{k(Y)}$:
\begin{equation} \label{eqn:valuation_of_diff_form}
	\ord_P(\omega) = e_{X/Y, P} \cdot \ord_{\pi(P)}(\omega) + d_{X/Y, P}
\end{equation}
(see e.g. \cite[Proposition IV.2.2~(b)]{Hartshorne1977}). For any sheaf $\mc F$ on $X$ and $Q \in Y(k)$ we abbreviate $(\pi_* \mc F)_Q$ to $\mc F_Q$. We write briefly $\tr_{X/Y}$ for the trace
\[
	\tr_{k(X)/k(Y)} : k(X) \to k(Y).
\]
Note that it induces a map
\[
	\Omega_{k(X)} \cong k(X) \otimes_{k(Y)} \Omega_{k(Y)} \to \Omega_{k(Y)},
\]
which we also denote by $\tr_{X/Y}$.
For a future use we note the following properties of trace:
\begin{itemize}
	\item For $f \in k(X)$ and $Q \in Y(k)$:
	\begin{equation} \label{eqn:valuation_of_trace}
		\tr_{X/Y}(f) \in \mf m_{Y, Q}^{\alpha},
	\end{equation}
	where $\alpha := \min \{ \lfloor (\ord_P(f)+d_{X/Y, P})/e_{X/Y, P} \rfloor : P \in \pi^{-1}(Q) \}$.
	For the proof see \cite[Lemma 1.4 (b)]{Kock_galois_structure}. 
	\item Let $S := \pi^{-1}(Q)$. Then:
	\begin{equation} \label{eqn:trace_and_diff_forms}
		\tr_{X/Y}(\Omega_{X, Q}) \subset \Omega_{Y, Q} \quad \textrm{ and }
		\tr_{X/Y}(\Omega_X(S)_{Q}) \subset \Omega_Y(\{ Q \})_Q.
	\end{equation}
	\item[] Indeed, the first part of~\eqref{eqn:trace_and_diff_forms} follows by the main result of~\cite{Zannier_traces_diff_forms}.
	The second part is immediate by using the first part and noting that for any $f \in k(X)$, $f \neq 0$:
	\[
	\tr_{X/Y}(df/f) = \frac{d(N_{X/Y} f)}{N_{X/Y} f},
	\]
	where $N_{X/Y} : k(X) \to k(Y)$ is the norm of the extension of fields $k(X)/k(Y)$.
	\item For any $\eta \in \Omega_{k(X)}$ and $Q \in Y(k)$:
	\begin{equation} \label{eqn:residue_and_trace}
		\sum_{P \in \pi^{-1}(Q)} \res_P(\eta) = \res_Q(\tr_{X/Y}(\eta))
	\end{equation}
	\item[] (see \cite[Proposition 1.6]{Hubl_residual_representation} or \cite[p.~154, $(R_6)$]{Tate_residues_differentials_curves}).
\end{itemize}
In the most part of the article we will assume that $k$, $X$, $Y$ and $G$ are as in Theorem~\ref{thm:main_thm}.
In this situation we adopt the following notation:
\begin{itemize}
	\item $B \subset Y(k)$ -- the set of branch points of $\pi$,
	
	\item $R \subset X(k)$ -- the set of ramification points of $\pi$,
	
	\item $U := Y \setminus B$, $V := \pi^{-1}(U)$,
	
	\item $e_P := e_{X/Y, P}$, $d_P := d_{X/Y, P}$, $d_P' := d_{X/Y, P}'$, $d_P'' := d_{X/Y, P}''$ for $P \in X(k)$,
	\item $X_P := X/G_{P, 0}$ is the quotient curve. We denote the image of $P$ on $X_P$ by $\ol P$.
	Similarly, for set $S \subset X(k)$, we write $\ol S$ for its image on $X_P$.
\end{itemize}
Also, by abuse of notation, for $Q \in Y(k)$ we write $G_{Q, i} := G_{P, i}$, $e_Q := e_P$, $d_Q := d_P$, $X_Q := X_P$ etc. for any $P \in \pi^{-1}(Q)$. Note that these quantities don't depend on the choice
of $P$. \\

Recall that the map:
\[
k(X) \times G \to k(X), \qquad f \cdot g := g^*(f)
\]
induces a natural right action on $k(X)$, since $g_1^* \circ g_2^* = (g_2 \cdot g_1)^*$ for any $g_1, g_2 \in G$.
Similarly, we have the structure of a right $k[G]$-module on $H^0(X, \Omega_X)$, $H^1(X, \mc O_X)$, $H^1_{dR}(X)$, etc. Note also that $G_P = G_{P, 0} = G_{P, 1}$, since $k$ is algebraically closed of characteristic $p$ and $G$ is a $p$-group (cf.~\cite[Corollary 4.2.3., p. 67]{Serre1979}).

\section{Magical elements} \label{sec:magical_elements}
Keep the assumptions of Theorem~\ref{thm:main_thm}. In this section we study the properties of the magical element $z \in k(X)$
satisfying the condition~\ref{enum:B}.\\
It turns out that the condition $\tr_{X/Y}(z) \neq 0$ guarantees that $z$ is a normal element,
see e.g.~\cite[Theorem 1]{Childs_Orzech_On_modular}. We give a different proof for completeness.
\begin{Proposition} \label{prop:g(z)_is_a_basis}
	The set $\{ g^*(z) : g \in G \}$ is a $k(Y)$-basis of $k(X)$.
\end{Proposition}
\begin{proof}
	The proof is based on the following identity in the ring $\FF_p[x_g : g \in G]$ ("Group Determinant Formula"):
	\begin{equation} \label{eqn:G-determinant}
		\det [x_{gh}]_{g, h \in G} = \pm \left(\sum_{g \in G} x_g \right)^{\# G}.
	\end{equation}
	This formula follows from \cite[Theorem~7]{Formanek_Sibley_gp_det}.
	Indeed, in the notation of \cite{Formanek_Sibley_gp_det}, if $G$ is a $p$-group, one has $G = \mc O_p(G)$, $y_{\ol g} = \sum_{g \in G} x_g$ and
	$\mc D_G(x_g) = \det[x_{g h^{-1}}]_{g, h \in G} = \pm \det [x_{gh'}]_{g, h' \in G}$ (where $h' = h^{-1}$ and $\pm$ is the sign of the permutation $h \mapsto h^{-1}$ of~$G$).
	See also~\cite[Lemma 2.2 and Corollary 2.4]{Huynh_Artin_Schreier_extensions}
	and~\cite[Lemma~5.26(a)]{Washington_Intro_to_cyclotomic} for versions of this formula for abelian groups.\\

	We show that the set $\{ g^*(z) : g \in G \}$ is linearly independent over $k(Y)$.
	Suppose that for some $f_g \in k(Y)$:
	\begin{equation*}
		0 = \sum_{g \in G} g^*(z) \cdot f_g.
	\end{equation*}
	Then for any $h \in G$:
	\begin{equation*}
		0 = \sum_{g \in G} h^*(g^*(z)) \cdot f_g = \sum_{g \in G} (gh)^*(z) \cdot f_g.
	\end{equation*}
	However, by~\eqref{eqn:G-determinant}:
	\[
	\det[(gh)^*(z)] = \pm \tr_{X/Y}(z)^{\# G} \neq 0
	\]
	and hence $f_g = 0$ for all $g \in G$. This ends the proof.
\end{proof}
Note that $\Omega_{k(X)}$ is a rank one $k(X)$-module and similarly $\Omega_{k(Y)}$ is a rank one $k(Y)$-module.
Hence, in the light of Proposition~\ref{prop:g(z)_is_a_basis}, for any $\omega \in \Omega_{k(X)}$ there exists
a unique system of differential forms $(\omega_g)_{g \in G}$ in $\Omega_{k(Y)}$ for which
\[
\omega = \sum_{g \in G} g^*(z) \omega_g.
\]
Note that for any $g, h \in G$:
\begin{equation} \label{eqn:g_component_of_h}
	h^*(\omega)_g = \omega_{g h^{-1}}.
\end{equation}
Indeed:
\begin{align*}
	h^*(\omega) = \sum_{g \in G} h^*(g^*(z)) \omega_g = \sum_{g \in G} (g \cdot h)^*(z) \omega_g
	= \sum_{g \in G} g^*(z) \omega_{gh^{-1}}.
\end{align*}
For any $g \in G$ we define the sheaf $g^*(z) \Omega_Y$ as the image of the subsheaf $\Omega_Y$ of the constant sheaf $\pi_* \Omega_{k(X)}$ under multiplication
by $g^*(z) \in k(X)$. We use similar convention for $\Omega_Y(B)$.
\begin{Lemma} \label{lem:inclusions_of_modules}	
	We have the following inclusions of sheaves on $Y$:
	\begin{align*}
		\bigoplus_{g \in G} g^*(z) \Omega_Y &\subset \pi_* \Omega_X, \\
		\bigoplus_{g \in G} g^*(z) \Omega_Y(B) &\subset \pi_* \Omega_X(R).
	\end{align*}
\end{Lemma}
\begin{proof}
	Note that by~\eqref{eqn:valuation_of_diff_form} for any $Q \in Y(k)$, $\omega \in \Omega_{Y, Q}$, $g \in G$ and $P \in \pi^{-1}(Q)$:
	\begin{equation*}
		\ord_P(g^*(z) \omega) \ge -d_P' + e_P \cdot \ord_Q(\omega) + d_P \ge -d_P' + d_P = e_P - 1 \ge 0.
	\end{equation*}
	The first inclusion follows. The second inclusion may be proven analogously.
\end{proof}
Observe that $\pi_* \Omega_X(R)$ and $\bigoplus_{g \in G} g^*(z) \Omega_Y(B)$ are coherent sheaves
of rank $\# G$. Therefore, their quotient is torsion and thus is isomorphic to:
\begin{equation} \label{eqn:quotient_differentials}
	\pi_* \Omega_X(R)/\bigoplus_{g \in G} g^*(z) \Omega_Y(B) \cong
	\bigoplus_{Q \in Y(k)} i_{Q, *}(H^0_Q) 
\end{equation}
where $i_Q : \Spec \mc O_{Y, Q} \to Y$ and:
\begin{equation} \label{eqn:H0Q}
	H^0_Q := \Omega_X(R)_Q/\bigoplus_{g \in G} g^*(z) \Omega_Y(B)_Q.
\end{equation}
\begin{Lemma} \label{lem:zQ_regular}
	$z_Q := \tr_{X/X_Q}(z)  \in \mc O_{X_Q, Q}$.
\end{Lemma}
\begin{proof}
	Fix a point $P \in \pi^{-1}(Q)$. Observe that $X_Q \to Y$ is unramified over~$Q$. Hence by~\eqref{eqn:valuation_of_trace}:
	\[
		\tr_{X/X_Q}(z) \in \mf m_{X_Q, \ol P}^{\lfloor (-d_P' + d_P)/e_P \rfloor} = \mc O_{X_Q, \ol P}.
	\]
	This finishes the proof.
\end{proof}
Recall that the dual basis of $\{ g^*(z) : g \in G \}$ with respect to the trace map is of
the form $\{ g^*(z^{\vee}) : g \in G \}$ for some $z^{\vee} \in k(X)$ (cf.~\cite[Theorem~3.13.19]{Hachenberger_Jungnickel_Topics}). By definition, $z^{\vee}$ satisfies
for any $g_1, g_2 \in G$:
\begin{equation} \label{eqn:def_of_dual_elt}
	\tr_{X/Y}(g_1(z) \cdot g_2(z^{\vee})) = 
	\begin{cases}
		1, & g_1 = g_2,\\
		0, & g_1 \neq g_2.
	\end{cases}
\end{equation}
For a future use note also that
\begin{equation} \label{eqn:trace_of_dual_z}
	\tr_{X/Y}(z), \, \tr_{X/Y}(z^{\vee}) \in k^{\times}.
\end{equation}
Indeed, by~\eqref{eqn:valuation_of_trace} for every $Q \in Y(k)$ one has $\tr_{X/Y}(z) \in \mc O_{Y, Q}$.
This yields $\tr_{X/Y}(z) \in \bigcap_{Q \in Y(k)} \mc O_{Y, Q} = H^0(Y, \mc O_Y) = k$ and 
$\tr_{X/Y}(z) \in k^{\times}$. Moreover by~\eqref{eqn:def_of_dual_elt}:
\begin{align*}
	\tr_{X/Y}(z^{\vee}) 
	&= \frac{\tr_{X/Y} \left(z^{\vee} \cdot \tr_{X/Y}(z) \right)}{\tr_{X/Y}(z)}
	= \frac{\tr_{X/Y} \left(z^{\vee} \cdot \sum_{g \in G} g^*(z) \right)}{\tr_{X/Y}(z)}\\
	&= \frac{\sum_{g \in G} \tr_{X/Y}(z^{\vee} \cdot g^*(z))}{\tr_{X/Y}(z)} = \frac{1}{\tr_{X/Y}(z)} \in k^{\times}.
\end{align*}	
For any $f \in k(X)$, we will denote by $(f_g)_{g \in G}$ the unique system of functions $f_g \in k(Y)$ such that:
\[
f = \sum_{g \in G} g^*(z^{\vee}) f_g.
\]
Note that by~\eqref{eqn:def_of_dual_elt} for any $g \in G$, $\omega \in \Omega_{k(X)}$ and $f \in k(X)$:
\begin{align}
	\tr_{X/Y}(g^*(z^{\vee}) \cdot \omega) = \omega_g \quad \textrm{ and } \quad
	\tr_{X/Y}(g^*(z) \cdot f) = f_g. \label{eqn:gth_component_trace_f}
\end{align}
For any $g \in G$ define the sheaves $g^*(z^{\vee}) \mc O_Y$ and $g^*(z^{\vee}) \mc O_Y(-B)$ similarly as  $g^*(z) \Omega_Y$ and
 $g^*(z) \Omega_Y(B)$.
\begin{Lemma} \label{lem:inclusions_of_modules2}
	We have the following inclusions:
	\begin{align*}
		\pi_* \mc O_X &\subset \bigoplus_{g \in G} g^*(z^{\vee}) \mc O_Y,\\
		\pi_* \mc O_X(-R) &\subset \bigoplus_{g \in G} g^*(z^{\vee}) \mc O_Y(-B).
	\end{align*}
\end{Lemma}
\noindent Before the proof, note that for any $Q \in Y(k)$ the pairing:
\begin{equation} \label{eqn:local_duality}
	(\cdot, \cdot) : k(Y)/\mc O_{Y, Q} \times \Omega_{Y, Q} \to k, \qquad (f, \omega) \mapsto \res_Q(f \cdot \omega)
\end{equation}
is non-degenerate. Indeed, if $\omega \in \Omega_{Y, Q}$, $\omega \neq 0$
then $(t^{-n-1}, \omega) \neq 0$ where $n := \ord_Q(\omega)$ and $t \in k(Y)$ is a uniformizer in $Q$. Similarly, if $f \in k(Y)$, $\ord_Q(f) = n < 0$ then
$(f, t^{-n - 1} \, dt) \neq 0$.
\begin{proof}[Proof of Lemma~\ref{lem:inclusions_of_modules2}]
	Fix a point $Q \in X(k)$. Suppose that $f \in \mc O_{X, Q}$. Then, for any $\omega \in \Omega_{Y, Q}$, using~\eqref{eqn:gth_component_trace_f} and~\eqref{eqn:residue_and_trace}:
	\begin{align*}
		\res_Q(f_g \cdot \omega) &= \res_Q(\tr_{X/Y}(g^*(z) f) \cdot \omega)\\
		&= \sum_{P \in \pi^{-1}(Q)} \res_P(f  \cdot  g^*(z) \cdot \omega)
		= 0,
	\end{align*}
	where the last equality follows, since $f \in \mc O_{X, Q}$ and $g^*(z) \cdot \omega \in \Omega_{X, Q}$ by Lemma~\ref{lem:inclusions_of_modules}. Hence $f_g \in \mc O_{Y, Q}$.
	The second inclusion follows analogously.
\end{proof}
Lemma~\ref{lem:inclusions_of_modules2} implies that:
\begin{equation*} \label{eqn:quotient_functions}
	\frac{\bigoplus_{g \in G} g^*(z^{\vee}) \mc O_Y(-B)}{\pi_* \mc O_X(-R)} \cong \bigoplus_{Q \in Y(k)} i_{Q, *}(H^1_Q),
\end{equation*}
where:
\begin{equation} \label{eqn:H1Q}
	H^1_Q := \bigoplus_{g \in G} g^*(z^{\vee}) \mc O_Y(-B)_Q/\mc O_X(-R)_Q. 
\end{equation}
\begin{Lemma} \label{lem:properties_H0Q_H1Q}
	Let $H^0_Q$, $H^1_Q$ be defined by~\eqref{eqn:H0Q} and~\eqref{eqn:H1Q}.
	\begin{enumerate}[(1)]
		\item $H^1_Q$ is dual to $H^0_Q$ as a $k[G]$-module.
		
		\item If $d_Q'' = 0$ then $H^0_Q = H^1_Q = 0$.
		
		\item $\sum_{Q \in B} \dim_k H^0_Q = \sum_{Q \in B} \dim_k H^1_Q = \sum_{Q \in B} \frac{1}{2} d_Q'' \cdot \#\pi^{-1}(Q)$.
	\end{enumerate}
\end{Lemma}
\begin{proof}
	(1) One checks that the duality pairing is induced by:
		\begin{align*}
			\Omega_X(R)_Q \times \bigoplus_{g \in G} g^*(z^{\vee}) \mc O_Y(-B)_Q &\to k,\\
			(\omega, f) &\mapsto \sum_{g \in G} \res_Q(\omega_g \cdot f_g).
		\end{align*}
		We omit the details.
		
	(2) Suppose that $d_Q'' = 0$. Then $d_Q = 2 \cdot (e_Q - 1)$ and $d_Q' = (e_Q - 1)$.
		Write:
		\begin{equation*}
			z^{\vee} = \sum_{g \in G} g^*(z) \cdot f_g \qquad \textrm{ for } f_g \in k(Y).
		\end{equation*}
		Then~\eqref{eqn:def_of_dual_elt} yields $f_g = \det[a_{h_1 h_2}]_{h_1, h_2 \in G}/\det[\tr_{X/Y}(h_1(z) \cdot h_2(z))]_{h_1, h_2 \in G}$, where:
		\[
		a_{h_1 h_2} :=
		\begin{cases}
			tr_{X/Y}(h_1(z) \cdot h_2(z)), & h_1 \neq g,\\
			\delta_{h_2 g}, & h_1 = g.	
		\end{cases}
		\]
		But~the Group Determinant Formula~\eqref{eqn:G-determinant} easily implies that
		\begin{align*}
			\det[\tr_{X/Y}(h_1(z) \cdot h_2(z))]_{h_1, h_2 \in G} &= 
			\det[\tr_{X/Y}((h_1 \cdot h_2')(z) \cdot z)]_{h_1, h_2' \in G}\\
			&= \pm \tr_{X/Y}(z)^{2 \cdot \# G} \in k^{\times}.
		\end{align*}
		Moreover, \eqref{eqn:valuation_of_trace} yields that 
		\[
		tr_{X/Y}(h_1(z) \cdot h_2(z)) \in \mf m_{Y, Q}^{\left \lfloor \frac{-2d'_Q + d_Q}{e_Q} \right \rfloor} = \mc O_{X, Q}.
		\]
		Hence $f_g \in \mc O_{X, Q}$ and $\ord_Q(z^{\vee}) \ge \ord_Q(z) = - d_Q'$.
		Thus
		if $f \in \mc O_Y(-B)_Q$, then for any $P \in \pi^{-1}(Q)$:
		\[
		\ord_P(g^*(z^{\vee}) \cdot f) \ge -d_Q' + e_Q = 1
		\]
		and $g^*(z^{\vee}) \cdot f \in \mc O_X(-R)_Q$.
		It follows that $H^1_Q = 0$ and thus also $H^0_Q = 0$ by~(1).
		
	(3) This follows from~\eqref{eqn:quotient_differentials} by taking global sections and applying the Riemann--Hurwitz formula and Riemann--Roch theorem.
\end{proof}
We end this section by giving some necessary conditions for $\pi$ to have a magical element.
One of the conditions will play a role in the proof of Theorem~\ref{thm:main_thm} in Section~\ref{sec:dR}.
\begin{Lemma} \label{lem:GME_implies_no_etale_cover}
	Keep assumptions of Theorem~\ref{thm:main_thm}. Then:
	\begin{enumerate}[(1)]
		\item $\pi$ does not factor through an \'{e}tale morphism $X' \to Y$
		of degree $> 1$,
		
		\item $\langle G_Q : Q \in B \rangle = G$. 
	\end{enumerate}
\end{Lemma}
\begin{proof}
	\begin{enumerate}[(1)]
		\item Suppose to the contrary that $\pi$ factors through a non-trivial \'{e}tale morphism $X' \to Y$.
		Then $z$ is also a magical element for the cover $X \to X'$. Hence $\tr_{X/X'}(z) \in k^{\times}$.
		But then:
		\[
		\tr_{X/Y}(z) = \tr_{X'/Y}(\tr_{X/X'}(z)) = [k(X') : k(Y)] \cdot \tr_{X/X'}(z) = 0.
		\]
		Contradiction ends the proof.
		
		\item Note that $X' := X/\langle G_Q : Q \in B \rangle \to Y$ is an \'{e}tale subcover
		of $\pi$. Thus by~(1) it must be of degree $1$ and $G = \langle G_Q : Q \in B \rangle$.
	\end{enumerate}
\end{proof}

\section{Holomorphic differentials} \label{sec:OmegaX}
The goal of this section is to prove the part of Theorem~\ref{thm:main_thm} concerning $H^0(X, \Omega_X)$ and $H^1(X, \mc O_X)$. The first step in this direction is to compare holomorphic differentials and logarithmic differentials with poles in the ramification locus. This is achieved by the Proposition~\ref{prop:log_diffs_and_diffs} below. For any $Q \in B$ and $g \in G$, let $g_Q \in \bigoplus_{B} k[G]$ be the element with $g$ on the $Q$-th component and $0$ on other components. Define also $k[G]_B := \ker \left( \sum : \bigoplus_B k[G] \to k[G] \right)$ (note that $k[G]_B \cong k[G]^{\# B - 1}$ as a $k[G]$-module). 
\begin{Proposition} \label{prop:log_diffs_and_diffs}
	Keep assumptions of Theorem~\ref{thm:main_thm}. The map:
	\[
	\res_G : \Omega_{k(X)} \to \bigoplus_{Q \in B} k[G], \quad \omega \mapsto \sum_{Q \in B} \sum_{g \in G} \res_Q(\omega_g) \cdot g_Q
	\]	
	induces split $k[G]$-linear surjections $\res_G : H^0(X, \Omega_X) \to I_{X/Y}$ and $\res_G : H^0(X, \Omega_X(R)) \to k[G]_B$
	with the same kernel.
\end{Proposition}
We first show how Proposition~\ref{prop:log_diffs_and_diffs} implies the part of Theorem~\ref{thm:main_thm}
concerning $H^0(X, \Omega_X)$ and $H^1(X, \mc O_X)$.
\begin{proof}[Proof of Theorem~\ref{thm:main_thm}, part 1]
	Proposition~\ref{prop:log_diffs_and_diffs} implies that we have the following isomorphisms of right $k[G]$-modules:
	\begin{align}
		H^0(X, \Omega_X) \oplus k[G]^{\# B - 1} \label{eqn:H0Omega=H0OmegaR}
		&\cong \ker(\res_G : H^0(X, \Omega_X) \to I_{X/Y}) \oplus I_{X/Y} \oplus k[G]_B\\
		&\cong \ker(\res_G : H^0(X, \Omega_X(R)) \to k[G]_B) \oplus I_{X/Y} \oplus k[G]_B \nonumber \\
		&\cong H^0(X, \Omega_X(R)) \oplus I_{X/Y}. \nonumber
	\end{align}
	Recall that by~\eqref{eqn:quotient_differentials} we have an exact sequence of $\mc O_Y$-modules:
	\[
	0 \to \bigoplus_{g \in G} g^*(z) \Omega_Y(B) \to \pi_* \Omega_X(R) \to \bigoplus_{Q \in Y(k)} i_{Q, *}(H^0_Q) \to 0.
	\]
	Moreover, by Serre's duality and Riemann--Roch theorem (cf.~\cite[Corollary III.7.7, Theorem IV.1.3]{Hartshorne1977}):
	\begin{align*}
		\dim_k H^0(Y, \Omega_Y(B)) &= g_Y + \# B - 1,\\
		\dim_k H^1(Y, \Omega_Y(B)) &= 0.
	\end{align*}
	(note that for $\# G > 1$, we have $B \neq \varnothing$ by Lemma \ref{lem:GME_implies_no_etale_cover}~(1)). Hence, after taking sections:
	\[
	0 \to k[G]^{g_Y + \# B - 1} \to H^0(X, \Omega_X(R)) \to \bigoplus_{Q \in B} H^0_Q \to 0.
	\]
	Note that $k[G]$ is injective as a $k[G]$-module (see~\cite[Corollary 8.5.3]{Webb_finite_group_representations}). Hence:
	\[
	H^0(X, \Omega_X(R)) \cong k[G]^{g_Y + \# B - 1} \oplus \bigoplus_{Q \in B} H^0_Q.
	\label{eqn:structure_Omega_X_R}
	\]
	We combine this with~\eqref{eqn:H0Omega=H0OmegaR} to obtain:
	\begin{equation} \label{eqn:proof_of_MT1_before_dividing}
		H^0(X, \Omega_X) \oplus k[G]^{\# B - 1} \cong k[G]^{g_Y + \# B - 1} \oplus I_{X/Y} \oplus \bigoplus_{Q \in B} H^0_Q.
	\end{equation}
	Since every $k[G]$-module has a unique decomposition into indecomposable $k[G]$-modules (cf. \cite[Corollary~11.1.7.]{Webb_finite_group_representations}), we may subtract $k[G]^{\# B - 1}$ from both sides
	of~\eqref{eqn:proof_of_MT1_before_dividing}. In this way we obtain the
	part of Theorem~\ref{thm:main_thm} concerning $H^0(X, \Omega_X)$. Finally, by Serre's duality,
	$H^1(X, \mc O_X)$ is the dual of $H^0(X, \Omega_X)$. This immediately implies the part
	of Theorem~\ref{thm:main_thm} concerning the cohomology of the structure sheaf.
\end{proof}
The proof of Proposition~\ref{prop:log_diffs_and_diffs} will occupy the rest of this section.
In the sequel we will need the relative augmentation ideal $I_{G, H}$ (as defined in Section~1),
where $H$ is a normal subgroup of $G$. We identify it with:
\[
I_{G, H} = \left\{ \sum_{g \in G} a_g g \in k[G] : \sum_{g \in g_0 H} a_g = 0 \quad \forall_{g_0 \in G} \right\}.
\]

Recall that in order to prove Proposition~\ref{prop:log_diffs_and_diffs} we study the image
of $H^0(X, \Omega_X)$ under the map $\res_G$. Let $\omega \in H^0(X, \Omega_X)$. Suppose for simplicity that $\pi^{-1}(Q) = \{ P \}$ and
$\tr_{X/Y}(z) = 1$. Then by~\eqref{eqn:residue_and_trace} $\res_P(g^*(z) \omega_g) = \res_Q(\omega_g)$. Hence:
\begin{align*}
	\sum_{g \in G} \res_Q(\omega_g) = \res_P(\omega) = 0
\end{align*}
and $\sum_{g \in G} \res_Q(\omega_g) \cdot g \in I_G$. In general, if $G_Q \neq G$, $\sum_{g \in G} \res_Q(\omega_g) \cdot g \in I_{G, G_Q}$, which is a consequence of the following lemma.
\begin{Lemma} \label{lem:main_lemma_Omega_Y}
	Keep assumptions of Theorem~\ref{thm:main_thm} and let $Q \in B$.
	\begin{enumerate}[(1), labelindent=0pt, itemindent=0pt, labelwidth=!]
		\item For every $\omega \in \Omega_{k(X)}$ and $g_0 \in G$, the form
		\[
			\sum_{g \in g_0 G_Q} \omega_g
		\]
		can be expressed as a linear combination of forms $g^*(\tr_{X/X_Q}(\omega))$ for $g \in G$ with coefficients
		in $\mc O_{X_Q, Q}$.

		\item For every $\omega \in \Omega_X(R)_Q$ one has:
		\[
			\omega \in \Omega_{X, Q} \quad \Leftrightarrow \quad \forall_{g_0 \in G} \, \sum_{g \in g_0 G_Q} \res_Q(\omega_g) = 0.
		\]
	\end{enumerate}
\end{Lemma}
\begin{proof}
	Fix a point $P \in \pi^{-1}(Q)$. Let $G/G_Q = \{ g_1 G_Q, \ldots, g_r G_Q \}$ and $P_i := g_i(P)$.
	Define $\omega_i := \sum_{g \in g_i G_Q} \omega_g$ for $i = 1, \ldots, r$.
	
	(1) Observe that for any $g \in g_i G_Q$:
		\begin{align}
			\tr_{X/X_Q}(g^*(z)) = \sum_{h \in g G_Q} h^*(z) = \sum_{h \in g_i G_Q} h^*(z) = \sum_{h \in G_Q g_i} h^*(z) = g_i^*(z_Q) \label{eqn:tr_XQ_gz}
		\end{align}
		(we used the fact that $G_Q$ is normal in $G$). Hence:
		\[
		\tr_{X/X_Q}(\omega) = \sum_{g \in G} \tr_{X/X_Q}(g^*(z)) \cdot \omega_g = \sum_{i = 1}^r g_i^*(z_Q) \cdot \omega_i.
		\]
		This implies that:
		\begin{equation} \label{eqn:system_eqns_tr_eta}
			g_j^*(\tr_{X/X_Q}(\omega)) = \sum_{i = 1}^r (g_i \cdot g_j)^*(z_Q) \cdot \omega_i
		\end{equation}
		for every $j = 1, \ldots, r$. By the Group Determinant Formula~\eqref{eqn:G-determinant} for the group $G/G_Q$:
		\[
		\det[(g_i \cdot g_j)^*(z_Q)] = \pm \left(\sum_{i=1}^r g_i^*(z_Q) \right)^{\# G/G_Q} = \pm \tr_{X/Y}(z)^{\# G/G_Q} \in k^{\times}.
		\]
		Furthermore, since $z_Q \in \mc O_{X_Q, Q}$ and $G_Q \unlhd G$, we have $g_i^*(z_Q) \in \mc O_{X_Q, Q}$ for any $i = 1, \ldots, r$.
		Therefore the system of linear equations~\eqref{eqn:system_eqns_tr_eta} and Lemma~\ref{lem:zQ_regular} imply that $\omega_i$ can
		be expressed as a combination of forms $g_j^*(\tr_{X/X_Q}(\omega))$ with coefficients in $\mc O_{X_Q, Q}$.\\
		
	(2) If $\omega \in \Omega_{X, Q}$ then $\res_Q(\omega_i) = 0$ for $i = 1, \ldots, r$ by part~(1) and~\eqref{eqn:trace_and_diff_forms}.
	Suppose now that $\omega \in \Omega_X(R)_Q$ and $\res_Q(\omega_i) = 0$ for $i = 1, \ldots, r$.
		Note that $\tr_{X/X_Q}(\omega) \in \Omega_X(\ol{R})_{Q}$ by~\eqref{eqn:trace_and_diff_forms}. Therefore, by (1) we have
		$\omega_i \in \Omega_{X_Q}(\ol R)_Q \cap \Omega_{k(Y)} = \Omega_Y(B)_Q$ and
		$\res_Q(\omega_i) = 0$. Hence $\omega_i$ is holomorphic at $Q$ (it is a logarithmic form
		with vanishing residues). 
		It follows that:
		\[
		\tr_{X/X_Q}(\omega) = \sum_{i = 1}^r g_i^*(z_Q) \cdot \omega_i \in \Omega_{X_Q, Q}.
		\]
		Hence, using~\eqref{eqn:residue_and_trace}, $\res_{P_j}(\omega) = \res_{\ol{P}_j}(\tr_{X/X_Q}(\omega)) = 0$
		for every $j$. Therefore $\omega$ must be holomorphic.	
\end{proof}
\noindent For a future use, we prove now the following consequence of Lemma~\ref{lem:main_lemma_Omega_Y}~(2).
\begin{Corollary} \label{cor:main_lemma_OX}
	Keep assumptions of Theorem~\ref{thm:main_thm}. Then for any $Q \in Y(k)$, we have $z_Q^{\vee} := \tr_{X/X_Q}(z^{\vee}) \in \mc O_{X_Q, Q}$. Moreover, $z^{\vee} \in \mc O_X(V)$.
\end{Corollary}
\begin{proof}
	Before the proof note that for every $Q \in Y(k)$ the pairing:
	\begin{align} 
		k(X)/\mc O_{X, Q} \times \Omega_{X, Q} &\to k, \label{eqn:local_duality_X}\\
		(f, \omega) &\mapsto \sum_{P \in \pi^{-1}(Q)} \res_P(f \cdot \omega) \nonumber
		= \res_Q(\tr_{X/Y}(f \cdot \omega))
	\end{align}
	(note that we used~\eqref{eqn:residue_and_trace} in the second equality) 
	is non-degenerate. Indeed, for every $P \in \pi^{-1}(Q)$ one can choose $t_P \in k(X)$ such that $\ord_P(t_P) = 1$ and $\ord_{P'}(t_P) = 0$ for $P' \in \pi^{-1}(Q)$, $P' \neq P$.
	Then $(t_P^{-n-1}, \omega) =  \res_P(t_P^{-n-1} \cdot \omega) \neq 0$
	for $n = \ord_P(\omega)$. Similarly one proves that this pairing
	is non-degenerate on the left (cf. the proof of non-degeneracy of~\eqref{eqn:local_duality}).\\
	
	It suffices to show that $z_Q^{\vee} \in \mc O_{X, Q}$. To this end note that for any $\omega \in \Omega_{X, Q}$:
	\begin{align*}
		\res_Q(\tr_{X/Y}(z^{\vee}_Q \cdot \omega))
		= \sum_{g \in G_Q} \res_Q(\tr_{X/Y}(g^*(z^{\vee}) \cdot \omega))
		= \sum_{g \in G_Q} \res_Q(\omega_g) = 0
	\end{align*}
	by Lemma~\ref{lem:main_lemma_Omega_Y}~(2) and by~\eqref{eqn:gth_component_trace_f}.
	Thus $z_Q^{\vee} \in \mc O_{X, Q}$ by the non-degeneracy of~\eqref{eqn:local_duality_X}. For the second part, note that
	for any $Q \in U$, $z^{\vee} = z^{\vee}_Q \in \mc O_{X, Q}$. Hence $z^{\vee} \in \mc O_X(V)$.
\end{proof}
The following exact sequence enables to construct various differential forms on $Y$ from ``local data''
(cf. \cite[III.7, p.~248]{Hartshorne1977}):
\begin{align} \label{eqn:constructing_diff_forms}
	0 \to \Omega_Y(Y) \to \Omega_{k(Y)} \to \bigoplus_{Q \in Y(k)} \frac{\Omega_{k(Y)}}{\Omega_{Y, Q}}
	&\to \, \, k \to 0,\\
	(\omega_Q)_{Q \in Y(k)} &\mapsto \sum_{Q \in Y(k)} \res_Q(\omega_Q). \nonumber
\end{align}
Fix a point $Q_0 \in B$. For any $Q \in B$, $Q \neq Q_0$, let $\eta_Q \in H^0(Y, \Omega_Y(Q_0+Q))$ be a fixed
differential form satisfying:
\[
\res_Q(\eta_Q) = 1, \quad \res_{Q_0}(\eta_Q) = -1
\]
(note that such a form exists by~\eqref{eqn:constructing_diff_forms}). Denote also $\eta_{Q_0} = 0$.\\
\subsection*{Proof of {Proposition~\ref{prop:log_diffs_and_diffs}}}
The proof is divided into three steps.
We abbreviate $\sum_{Q \in B} \sum_{g \in G}$ to $\sum_{Q, g}$.
\subsection*{Step I} The map $\res_G$ defines $k[G]$-linear homomorphisms:
\[
H^0(X, \Omega_X) \to I_{X/Y} \quad \textrm{ and } \quad H^0(X, \Omega_X(R)) \to k[G]_B.
\]
\begin{proof}[Proof of Step I]
	We check now that $\res_G$ defines a map $H^0(X, \Omega_X) \to I_{X/Y}$. Let $\omega \in H^0(X, \Omega_X)$.
	Then for any $Q \in B$, $\sum_{g \in G} \res_Q(\omega_g) g \in I_{G, G_Q}$ by Lemma~\ref{lem:main_lemma_Omega_Y}~(2). 
	Moreover:
	\[
	\sum_{Q, g} \res_Q(\omega_g) g = \sum_{g \in G} g \sum_{Q \in B} \res_Q(\omega_g) = 0
	\]
	by the residue theorem (cf.~\cite[Corollary after Theorem 3]{Tate_residues_differentials_curves}). It follows that the image of $H^0(X, \Omega_X)$ is contained in $I_{X/Y}$. The $k$-linearity is easy to check. The map in question
	is $G$-equivariant, since for any $h \in G$, the form $\omega \cdot h = h^*(\omega)$ maps to:
	\begin{align*}
		\sum_{Q, g} \res_Q((h^*\omega)_g) \cdot g_Q
		&= \sum_{Q, g} \res_Q(\omega_{gh^{-1}}) \cdot g_Q\\
		&=\sum_{Q, g} \res_Q(\omega_g) \cdot (g \cdot h)_Q\\
		&= \left( \sum_{Q, g} \res_Q(\omega_{g}) \cdot g_Q \right) \cdot h
	\end{align*}
	(here we used~\eqref{eqn:g_component_of_h}).
	One proves that $\res_G$ defines a map $H^0(X, \Omega_X(R)) \to k[G]_B$ by applying residue theorem in a similar manner.
\end{proof}
\subsection*{Step II}
The map:
\[
\bigoplus_{Q \in B} k[G] \to \Omega_{k(X)}, \quad
\sum_{Q, g} a_{Q, g} g_Q \mapsto \sum_{Q, g} a_{Q, g} g^*(z) \eta_Q
\]
induces $k[G]$-linear homomorphisms:
\[
I_{X/Y} \to H^0(X, \Omega_X) \quad \textrm{ and } \quad k[G]_B \to H^0(X, \Omega_X(R)).
\]
These maps provide sections of the homomorphisms $\res_G : H^0(X, \Omega_X) \to I_{X/Y}$ and
$\res_G : H^0(X, \Omega_X(R)) \to k[G]_B$ respectively.
\begin{proof}[Proof of Step II]
	One easily checks that the considered map is $k[G]$-linear.
	Let $\sum_{Q, g} a_{Q, g} g_Q \in I_{X/Y}$ and $\omega := \sum_{Q, g} a_{Q, g} g^*(z) \eta_Q$. Then by definition $\omega \in H^0(X, \Omega_X(R))$ and $\omega_h = \sum_Q a_{Q, h} \eta_Q$ for any $h \in G$. Hence for any $Q' \in B$ and $h \in G$:
	\begin{align*}
		\res_{Q'} (\omega_h) &=
		\begin{cases}
			a_{Q', h}, & \textrm{ if } Q' \neq Q_0,\\
			-\sum_{Q \neq Q_0} a_{Q, h}, & \textrm{ if } Q' = Q_0,
		\end{cases}\\
		&=
		\begin{cases}
			a_{Q', h}, & \textrm{ if } Q' \neq Q_0,\\
			a_{Q_0, h}, & \textrm{ if } Q' = Q_0,
		\end{cases}
	\end{align*}
	(the last equality follows by the fact that $\sum_{Q, g} a_{Q, g} g_Q \in k[G]_B$). Thus
	\[
		\sum_{g \in G} \res_{Q}(\omega_g) \cdot g = \sum_{g \in G} a_{Q, g} \cdot g \in I_{G, G_Q}
	\]
	 for any $Q \in B$ and $\omega \in \Omega_{X, Q}$ by Lemma~\ref{lem:main_lemma_Omega_Y}~(2). This means that $\omega \in H^0(X, \Omega_X)$.
	The above reasoning shows also that $\res_G(\omega) = \sum_{Q, g} a_{Q, g} g_Q$, i.e. this map is a section of $\res_G : H^0(X, \Omega_X) \to I_{X/Y}$.
	One checks that the map $k[G]_B \to H^0(X, \Omega_X(R))$
	is a well-defined section of $\res_G : H^0(X, \Omega_X(R)) \to k[G]_B$ in a similar manner.
\end{proof}
\subsection*{Step III}
Note that by Lemma~\ref{lem:main_lemma_Omega_Y}:
\[
H^0(X, \Omega_X) = \{ \omega \in H^0(X, \Omega_X(R)) : \res_G(\omega) \in I_{X/Y} \}.
\]
This can be rephrased by saying that the commutative diagram:
	\begin{center}
	\begin{tikzcd}
		{H^0(X, \Omega_{X})} \arrow[d, "{\textrm{res}_G}", two heads] \arrow[r, hook] & {H^0(X, \Omega_{X}(R))} \arrow[d, "{\textrm{res}_G}", two heads] \\
		I_{X/Y} \arrow[r, hook]                                             & {k[G]_B}                                             
	\end{tikzcd}
\end{center}
is cartesian. Therefore vertical arrows have the same kernels (see e.g. \cite[\href{https://stacks.math.columbia.edu/tag/08N3}{Lemma 08N3~(1)}]{stacks-project}). This ends the proof. \qed

\section{Cohomology of the structure sheaf} \label{sec:OX}
By Serre's duality, the dual of the map $\res_G : H^0(X, \Omega_X) \to I_{X/Y}$ may be identified with a map
\[
	J_{X/Y} \cong I_{X/Y}^{\vee} \to H^0(X, \Omega_X)^{\vee} \cong H^1(X, \mc O_X).
\]
Analogously one obtains a map $k[G]_B^{\vee} \to H^1(X, \mc O_X(-R))$. The goal of this section is to describe explicitly these two maps. This will be needed in the proof of the part of Theorem~\ref{thm:main_thm} concerning the de Rham cohomology. In the sequel we will identify $H^1(X, \mc O_X)$ and $H^1(X, \mc O_X(-R))$ with certain quotients
of $\bigoplus_B k(X)$ (this is explained in Lemma~\ref{lem:description_of_sheaf_cohomology} below). We treat $\bigoplus_B k(X)$ as a $k(X)$-module.
For any $Q \in B$ denote by $\delta(Q)$ the element of $\bigoplus_{B} k(X)$ with $1$ on the $Q$-th component and zero
on other components. The following is the main result of this section.
\begin{Proposition} \label{prop:resGvee}
	Keep assumptions of Theorem~\ref{thm:main_thm}. The map:
	\begin{equation} \label{eqn:map_kG_kX}
		\bigoplus_{Q \in B} k[G] \to \bigoplus_{Q \in B} k(X), \quad
		g_Q \to g^*(z^{\vee}) \cdot \delta(Q)
	\end{equation}
	induces maps:
	\[
	\res_G^{\vee} : J_{X/Y} \to H^1(X, \mc O_X), \quad \res_G^{\vee} : k[G]_B^{\vee} \to H^1(X, \mc O_X(-R))
	\]
	dual to the maps $\res_G$.
\end{Proposition}
In order to express $H^1(X, \mc O_X)$ and $H^1(X, \mc O_X(-R))$ as quotients of $\bigoplus_B k(X)$, we will use an alternative description of sheaf cohomology of sheaves on a curve. Basically, it is a variation of \v{C}ech cohomology for a cover consisting of an open set $U$ and ``infinitesimal neighbourhoods'' of points $Q \not \in U(k)$.
\begin{Lemma} \label{lem:description_of_sheaf_cohomology}
	Let $Y$ be a smooth projective curve with the generic point~$\eta$ over an algebraically closed field $k$. Let $S \subset Y(k)$ be a finite non-empty set. Denote $U := Y \setminus S$. Then for any locally free sheaf $\mc F$ of finite rank on $Y$ we have a natural isomorphism:
	\begin{align*}
		H^1(Y, \mc F) &\cong \coker(\mc F(U) \to \bigoplus_{Q \in S} \mc F_{\eta}/\mc F_{Q}).
	\end{align*}
\end{Lemma}
\begin{proof}
	Let $j : U \hookrightarrow Y$ be the open immersion.
	For any $Q \in Y(k)$:
	\begin{align*}
		j_*(\mc F|_U)_Q = \varinjlim_{V \ni Q} \mc F (U \cap V) =
		\begin{cases}
			\mc F_Q, & \textrm{ if } Q \in U,\\
			\mc F_{\eta}, & \textrm{ otherwise, }
		\end{cases}
	\end{align*}
	(note that if $Q \in U(k)$, $U \cap V$ ranges over all open subsets of $U$ containing $Q$;
	otherwise $U \cap V$ ranges over all open subsets of $U$).
	This yields the exact sequence:
	\begin{equation} \label{eqn:F_to_FU_exact_sequence}
		0 \to \mc F \to j_*(\mc F|_U) \to \bigoplus_{Q \in S} i_{Q, *}(\mc F_{\eta}/\mc F_Q) \to 0,
	\end{equation}
	where $i_Q : \Spec(\mc O_{Y, Q}) \to Y$ is the natural morphism.
	Note that $U$ is an affine open subscheme by \cite[Ex. IV.1.3]{Hartshorne1977}.
	In particular, $j$ is an affine morphism. Thus, using \cite[Ex. III.4.1]{Hartshorne1977} and Serre's criterion on affineness (cf.~\cite[Theorem~III.3.7]{Hartshorne1977})
	we obtain $H^1(Y, j_*(\mc F|_U)) = H^1(U, \mc F|_U) = 0$.
	Therefore the proof follows by taking the associated long exact sequence of~\eqref{eqn:F_to_FU_exact_sequence}.
\end{proof}
Note that $H^1(X, \mc O_X) \cong H^1(Y, \pi_*\mc O_X)$ and $H^1(X, \mc O_X(-R)) \cong H^1(Y, \pi_*\mc O_X(-R))$. Therefore those two cohomology groups
might be considered as quotients of $\bigoplus_B k(X)$. By abuse of notation,
we often denote the element of the cohomology
determined by an element of $\bigoplus_B k(X)$ by the same letter.
For any $\nu \in \bigoplus_B k(X)$ and $Q \in B$ let $\nu_Q \in k(X)$ denote the $Q$-th coordinate of $\nu$. In this context, for any $D \in \Divv(X)$ the duality pairing
\[
\langle \cdot, \cdot \rangle : H^0(X, \Omega_X(D)) \times H^1(X, \mc O_X(-D)) \to k
\]
is given by the formula:
\begin{equation} \label{eqn:duality_pairing_formula}
	\langle \omega, \nu \rangle := \sum_{Q \in S} \res_Q(\tr_{X/Y}(\nu_Q \cdot \omega))
\end{equation}
(this can be proven using~\eqref{eqn:residue_and_trace}; see e.g. \cite[p.~158-159]{Tate_residues_differentials_curves} for similar formulas for the Serre's duality). In the sequel we identify $J_{G, H}$ with:
\[
J_{G, H} = k[G]\bigg/\left(\sum_{g \in G} a_g \cdot g : a_{g_1} = a_{g_2} \textrm{ if } g_1 H = g_2 H \right).
\]
Let $k[G]_B^{\vee}$ be the dual of the module $k[G]_B$. Note that $k[G]_B^{\vee}$ may be identified with the module:
\[
\coker \left(\textrm{diag} : k[G] \to \bigoplus_B k[G] \right).
\]
and that $k[G]_B^{\vee} \cong k[G]^{\# B - 1}$. For any $g \in G$ and $Q \in B$, denote the image of $g_Q$ in $J_{X/Y}$ and in $k[G]_B^{\vee}$ by $\ol{g}_Q$.
\begin{proof}[Proof of Proposition~\ref{prop:resGvee}]
	Again, we abbreviate $\sum_{Q \in B} \sum_{g \in G}$ to $\sum_{Q, g}$. Note that the duality pairing between $I_{X/Y}$ and $J_{X/Y}$ is given by:
	\begin{equation*}
		\left\langle \sum_{Q, g} a_{Q, g} g_Q, \sum_{Q, g} b_{Q, g} \ol g_Q  \right\rangle = \sum_{Q, g} a_{Q, g} b_{Q, g}.
	\end{equation*}
	For any $\omega \in H^0(X, \Omega_X)$, $\sum_{Q, g} b_{Q, g} g_Q \in \bigoplus_B k[G]$,
	using the duality~\eqref{eqn:duality_pairing_formula} and~\eqref{eqn:gth_component_trace_f}:
	\begin{align*}
		\left\langle \omega, \res_G^{\vee}\left(\sum_{Q, g} b_{Q, g} g_Q \right) \right\rangle
		&= \left\langle \omega, \sum_{Q, g} b_{Q, g} g^*(z^{\vee}) \delta(Q) \right\rangle \\
		&= \sum_{Q, g} b_{Q, g} \res_Q(\tr_{X/Y}(g^*(z^{\vee}) \cdot \omega)) \\
		&=  \sum_{Q, g} b_{Q, g} \res_Q(\omega_g) \\
		&= \left\langle \sum_{Q, g} \res_Q(\omega_g) g_Q, \sum_{Q, g} b_{Q, g} \ol g_Q \right\rangle \\
		&= \left\langle \res_G(\omega), \sum_{Q, g} b_{Q, g} \ol g_Q \right\rangle.
	\end{align*}
	This shows that the map~\eqref{eqn:map_kG_kX} induces a map $J_{X/Y} \to H^1(X, \mc O_X)$,
	which is dual to $\res_G : H^0(X, \Omega_X) \to I_{X/Y}$. The proof for
	$\res_G^{\vee} : k[G]_B^{\vee} \to H^1(X, \mc O_X(-R))$ is analogous.
	\end{proof}
\begin{Remark}
Note since $\res_G$ are split surjections, $\res_G^{\vee}$ are split injections by duality. Moreover, the diagram:
\begin{center}
	\begin{tikzcd}
	{H^1(X, \mc O_X(-R))} \arrow[r, two heads]    & {H^1(X, \mc O_X)}       \\
	{k[G]_B} \arrow[u, hook] \arrow[r, two heads] & J_{X/Y} \arrow[u, hook]
	\end{tikzcd}
\end{center}
is cocartesian.
\end{Remark}

\section{The de Rham cohomology} \label{sec:dR}
In this section we prove the part of Theorem~\ref{thm:main_thm} concerning the de Rham cohomology.
The idea is similar as in the case of $H^0(X, \Omega_X)$. First we compare the de Rham cohomology to the logarithmic de Rham cohomology with poles in the ramification locus. Then we compare the logarithmic de Rham cohomology
to its variant, which we call the ``quasilogarithmic de Rham cohomology''. Finally, we decompose
the quasilogarithmic de Rham cohomology into local and global parts.\\

Let $Y$ and $S$ be as in Lemma~\ref{lem:description_of_sheaf_cohomology}. Recall that $H^1_{dR, log}(Y, S)$, the logarithmic
cohomology of $Y$ with poles in $S$, is defined as the hypercohomology of the complex:
\begin{equation*}
	\Omega_Y^{\bullet}(\log S) : \quad (\mc O_Y \stackrel{d}{\longrightarrow} \Omega_Y(S)).
\end{equation*}
We define the quasilogarithmic de Rham cohomology with poles in $S$, $H^1_{dR, qlog}(Y, S)$,
as the hypercohomology of the complex:
\begin{equation*}
	\Omega_Y^{\bullet}(\qlog S) : \quad (\mc O_Y(-S) \stackrel{d}{\longrightarrow} \Omega_Y(S)).
\end{equation*}
Note that the spectral sequences of hypercohomology for $\Omega_Y^{\bullet}(\log S)$ and
$\Omega_Y^{\bullet}(\qlog S)$ degenerate on the first page for trivial reasons (since $d : k = H^0(Y, \mc O_Y) \to H^0(Y, \Omega_Y(S))$ vanishes and $H^1(Y, \Omega_Y(S)) = H^0(Y, \mc O_Y(-S)) = 0$).
This yields ``Hodge--de Rham exact sequences'' for $H^1_{dR, log}(Y, S)$ and $H^1_{dR, qlog}(Y, S)$:
\begin{alignat}{5}
	0 \to& H^0(Y, \Omega_Y(S)) &\to H^1_{dR, log}(Y, S) \to& \quad H^1(Y, \mc O_Y) &\to 0, \label{eqn:hdr_exact_sequence_for_log_de_rham} \\
	0 \to& H^0(Y, \Omega_Y(S)) &\to H^1_{dR, qlog}(Y, S) \to& H^1(Y, \mc O_Y(-S)) &\to 0. \label{eqn:hdr_exact_sequence_for_qlog_de_rham}
\end{alignat}
The long exact sequence associated with the exact sequence:
\[
0 \to \Omega_Y^{\bullet}(\qlog S) \to \Omega_Y^{\bullet}(\log S) \to \mc O_Y/\mc O_Y(-S) \to 0
\]
(where we treat $\mc O_Y/\mc O_Y(-S)$ as a complex concentrated in degree $0$)
provides a natural injection $H^1_{dR}(Y) \to H^1_{dR, log}(Y, S)$. Similarly, there is a natural
surjection $H^1_{dR, qlog}(Y, S) \to H^1_{dR, log}(Y, S)$.\\

The following result is an analogue of Propositions~\ref{prop:log_diffs_and_diffs} and~\ref{prop:resGvee}.
\begin{Proposition} \label{prop:extending_resG}
	Keep assumptions of Theorem~\ref{thm:main_thm}.
	\begin{enumerate}[(1)]
		\item The maps $\res_G : H^0(X, \Omega_X) \to I_{X/Y}$ and $\res_G : H^0(X, \Omega_X(R)) \to k[G]_B$ extend to split $k[G]$-surjections:
		\[
		\res_G : H^1_{dR}(X) \to I_{X/Y} \quad \textrm{ and } \quad \res_G : H^1_{dR, log}(X, R) \to k[G]_B
		\]
		with isomorphic kernels.
		
		\item The maps $\res_G^{\vee} : J_{X/Y} \to H^1(X, \mc O_X)$ and $\res_G^{\vee} : k[G]_B^{\vee} \to H^1(X, \mc O_X(-R))$ lift to split $k[G]$-injections:
		\[
		\res_G^{\vee} : J_{X/Y} \to H^1_{dR, log}(X, R) \quad \textrm{ and } \quad \res_G^{\vee} : k[G]_B^{\vee} \to H^1_{dR, qlog}(X, R)
		\]
		with isomorphic cokernels.
	\end{enumerate}
\end{Proposition}
\noindent 
We give now a description of sheaf hypercohomology, that generalizes Lemma~\ref{lem:description_of_sheaf_cohomology}.
\begin{Lemma} 
	Keep the setup of Lemma~\ref{lem:description_of_sheaf_cohomology}. Let $\mc F^{\bullet} = (\mc F^0 \stackrel{d}{\rightarrow} \mc F^1)$
	be a cochain complex of locally free $\mc O_Y$-modules of finite rank with a $k$-linear differential.
	Then we have a natural isomorphism:
	\begin{align*}
		\HH^1(Y, \mc F^{\bullet}) &\cong Z^1_S(\mc F^{\bullet})/B^1_S(\mc F^{\bullet}),
	\end{align*}
	where:
	\begin{align*}
		Z^1_{S}(\mc F^{\bullet}) &:= \{ (\omega, (h_Q)_{Q \in S} ) : \omega \in \mc F^1(U), 
		h_Q \in \mc F^0_{\eta}, \, \omega - d h_Q \in \mc F^1_Q \},\\
		B^1_{S}(\mc F^{\bullet}) &:= \{ (dh, (h + h_Q)_{Q \in S} ) : h \in \mc F^0(U), \,
		h_Q \in \mc F^0_Q \}.
	\end{align*}
	Moreover, the natural maps
	\begin{align*}
		H^0(Y, \mc F^1) \to \HH^1(Y, \mc F^{\bullet}) \quad \textrm{ and } \quad 
		\HH^1(Y, \mc F^{\bullet}) \to H^1(Y, \mc F^0)
	\end{align*}
	are induced by the maps:
	\[
	\begin{array}{ccc}
		H^0(Y, \mc F^1) &\to& Z^1_S(\mc F^{\bullet})\\
		\omega &\mapsto& (\omega, (0)_{Q \in S})
	\end{array}
	\qquad \textrm{ and } \qquad
	\begin{array}{ccc}
		Z^1_S(\mc F^{\bullet}) &\to& \bigoplus_{Q \in S} \mc F^0_{\eta}\\
		(\omega, (h_Q)_{Q \in S} ) &\mapsto& (h_Q)_{Q \in S}
	\end{array}
	\]
	respectively.
\end{Lemma}
\begin{proof}
	Before the proof we discuss our notation concerning cochain complexes. For any object $A$ we denote by
	$A[i]^{\bullet}$ the cochain complex concentrated in degree~$i$ with $A[i]^i = A$. Given a cochain complex $C^{\bullet}$,
	let $C[i]^{\bullet}$ denote its shift by $i$, i.e. $C[i]^j = C^{i+j}$ for all $j$. We denote by $h^i(C^{\bullet})$ the $i$-th cohomology
	of the cochain complex $C^{\bullet}$. 
	
	The exact sequence~\eqref{eqn:F_to_FU_exact_sequence} easily implies that $\mc F^{\bullet}$ is the kernel of the map of complexes:
	\[
	j_*(\mc F^{\bullet}|_U) \to \bigoplus_{Q \in S} i_{Q, *}(\mc F^{\bullet}_{\eta}/\mc F^{\bullet}_Q).
	\]
	Let $C^{\bullet}$ be the mapping cone of this map, i.e. the total complex of the double complex:
	\begin{center}
		\begin{tikzcd}
			j_*(\mc F^0|_U) \arrow[r] \arrow[d] & {\bigoplus_{Q \in S} i_{Q, *}(\mc F^0_{\eta}/\mc F^0_Q)} \arrow[d] \\
			j_*(\mc F^1|_U) \arrow[r]           & {\bigoplus_{Q \in S} i_{Q, *}(\mc F^1_{\eta}/\mc F^1_Q).}
		\end{tikzcd}
	\end{center}
	Then $\mc F^{\bullet}$ is quasi-isomorphic to $C[-1]^{\bullet}$. In particular, $\HH^1(Y, \mc F^{\bullet}) \cong \HH^1(Y, C^{\bullet}[-1])$.
	Consider the hypercohomology spectral sequence $E^{ij}_1 = H^j(Y, C^{\bullet}[-1]^i) \Rightarrow \HH^{i+j}(Y, C^{\bullet}[-1])$.
	Note that by a similar argument as in Lemma~\ref{lem:description_of_sheaf_cohomology}, $E^{ij}_1 = 0$ for $j \ge 1$. Thus $\HH^i(Y, C^{\bullet}[-1])
	= h^i(H^0(Y, C^{\bullet}[-1])) = h^{i-1}(H^0(Y, C^{\bullet}))$. In particular:
	\[
	\HH^1(Y, C^{\bullet}[-1]) = h^0(H^0(Y, C^{\bullet})) = Z^1_S(\mc F^{\bullet})/B^1_S(\mc F^{\bullet}).
	\]
	(the second equality follows directly from the definition of $H^0(Y, C^{\bullet})$).
	The second statement follows from the functoriality for the maps $\mc F^1[1]^{\bullet} \to \mc F^{\bullet}$
	and $\mc F^{\bullet} \to \mc F^0[0]^{\bullet}$.
\end{proof}
The key idea of the proof of Proposition~\ref{prop:extending_resG} is to decompose the differential form $dz^{\vee}$ into certain ``local components''.
We give now a brief motivation for this concept.
Suppose that we lifted the map $\res_G^{\vee} : J_{X/Y} \to H^1(X, \mc O_X)$
to a map $\res_G^{\vee} : J_{X/Y} \to H^1_{dR, log}(X, R)$. Let $e$ be the neutral element of $G$. Then $\res_G^{\vee}(\ol e_Q)$ must be of the form $(\xi_Q, z^{\vee} \cdot \delta(Q))$, where $\xi_Q \in \Omega_X(V)$ is a certain differential form.
The conditions defining $J_{X/Y}$ easily imply among others that $\sum_{Q \in B} \xi_Q = dz^{\vee}$. It turns out that the forms $(\xi_Q)_{Q \in B}$ are used not only in the construction
of $\res_G^{\vee} : J_{X/Y} \to H^1_{dR, log}(X, R)$, but also of $\res_G : H^1_{dR}(X) \to I_{X/Y}$.
The following lemma assures the existence of the forms $(\xi_Q)_{Q \in B}$ with the desired properties.
\begin{Lemma} \label{lem:xiQ_existence}
	There exists a system of differential forms $(\xi_Q)_{Q \in B}$, $\xi_Q \in \Omega_{k(X)}$
	such that the following conditions are satisfied:
	\begin{enumerate}[(1)]
		\item $\xi_Q \equiv dz^{\vee} \pmod{\Omega_{X, Q}}$,
		\item $\xi_Q \in \Omega_{X, Q'}$ for every $Q' \in Y(k)$, $Q' \neq Q$,
		\item $\sum_{g \in g_0 G_Q} g^*(\xi_Q) = 0$ for every $g_0 \in G$,
		\item $\sum_{Q \in B} \xi_Q = dz^{\vee}$.
	\end{enumerate}
\end{Lemma}
\begin{proof}
	Denote $J' := \bigoplus_{Q \in B} J_{G, G_Q}$. We divide the proof into four steps.\\
	
	\noindent \bb{Step I:} we define the homomorphism $\psi : J' \to \bigoplus_{Q \in B} \Omega_{k(Y)}/\Omega_Y(B)_Q$.\\
	\emph{Proof of Step I:} Consider the $k$-linear homomorphism:
	\begin{align*}
		\wt{\psi} : \bigoplus_{Q \in B} k[G]  &\to \bigoplus_{Q \in B} \Omega_{k(Y)}/\Omega_Y(B)_Q,\\
		g_Q &\mapsto (dz^{\vee})_g \cdot \delta(Q).
	\end{align*}
	Note that $\wt{\psi}$ induces a map:
	\[
	\psi : J' \to \bigoplus_{Q \in B} \Omega_{k(Y)}/\Omega_Y(B)_Q.
	\]
	Indeed, for any $g_0 \in G$, $Q \in B$:
	\begin{equation} \label{eqn:chi_is_regular}
		\sum_{g \in g_0 G_Q} (dz^{\vee})_g \in \Omega_{X_Q, Q} \cap \Omega_{k(Y)} = \Omega_{Y, Q},
	\end{equation}
	since $\sum_{g \in g_0 G_Q} (dz^{\vee})_g$ may be written as a combination of forms $h^*(dz_Q^{\vee})$ for $h \in G$ with coefficients in $\mc O_{X_Q, Q}$ (cf. Lemma~\ref{lem:main_lemma_Omega_Y}~(1) and Corollary~\ref{cor:main_lemma_OX}). Therefore:
	\begin{align*}
		\wt{\psi} \left(\sum_{g \in g_0 G_Q} g_Q \right) &= \sum_{g \in g_0 G_Q} (dz^{\vee})_g \cdot \delta(Q)\\
		&= 0 \qquad \textrm{ in } \bigoplus_{Q \in B} \Omega_{k(Y)}/\Omega_Y(B)_Q
	\end{align*}
	and $\psi$ is well-defined.\\
	
	\noindent \bb{Step II:} let $d_g := \sum_{Q \in B} \ol g_Q \in J'$ for any $g \in G$, $\mc D := \Span_k(d_g : g \in G)$ and:
	\[
		E_Q := \Span_k \left( \sum_{g \in g_0 G_Q} \ol g_{Q'} : g_0 \in G, Q' \in B \right) \subset J' \quad \textrm{ for any } Q \in B.
	\]
	Then $(d_g)_{g \neq e}$ is a basis of $\mc D$ and $E_Q \cap \mc D = \Span_k \left(\sum_{g \in g_0 G_Q} d_g : g_0 \in G \right)$.\\
	
	\emph{Proof of Step II:} Note that $\sum_{g \in G} d_g = 0$. It suffices to prove that that the elements $(d_g)_{g \neq e}$ are linearly independent over $k$.
	Suppose to the contrary that $\sum_{g \in G} a_g \cdot d_g = 0$ in $J'$ for some
	$a_g \in k$, where $a_e := 0$. Then by the definition of $J'$ we have $a_{g_1} = a_{g_2}$ for any $g_1, g_2 \in G$ for which there is a $Q \in B$ such that $g_1 G_Q = g_2 G_Q$. But by Lemma~\ref{lem:GME_implies_no_etale_cover}~(2)
	$G = \langle G_Q : Q \in B \rangle$. Hence any $g \in G$ can be written in the
	form $g_1 \cdot \ldots \cdot g_m$ for $g_i \in G_{Q_i}$,
	$Q_i \in B$. Therefore:
	\[
	0 = a_e = a_{g_1} = a_{g_1 g_2} = \ldots = a_g.
	\]
	Hence $(d_g)_{g \neq e}$ are linearly independent.\\
	Suppose now that $\sum_{g \in G} a_g d_g \in E_Q$ for some $Q \in B$ and $a_g \in k$, i.e.
	\[
		\sum_{g \in G} a_g d_g = \sum_{g_0, Q'} b_{g_0, Q'} \sum_{g \in g_0 G_Q} \ol g_{Q'}
	\]
	for some $b_{g_0, Q'} \in k$. Then (by lifting this relation to $\bigoplus_B k[G]$ and comparing the coefficients of $g_Q$) one sees that
	$a_{g_1} = a_{g_2}$ for any $g_1, g_2 \in G$ such that $g_1 G_Q = g_2 G_Q$. Thus, if $G/G_Q = \{ g_1 G_Q, \ldots, g_n G_Q \}$:
	\[
		\sum_{g \in G} a_g d_g = \sum_{i = 1}^n a_{g_i} \sum_{g \in g_i G_Q} d_g,
	\]
	which proves the second part.\\
	
	\noindent \bb{Step III:} we define the homomorphisms $\varphi_{\mc D} : \mc D \to \Omega_{k(Y)}$ and $\varphi : J' \to \Omega_{k(Y)}$.\\
	\emph{Proof of Step III:} Let the $k$-linear map $\varphi_{\mc D} : \mc D \to \Omega_{k(Y)}$ be defined by putting $\varphi_{\mc D}(d_g) := (dz^{\vee})_g$ for $g \neq e$. Note that then $\varphi(d_e) = (dz^{\vee})_e$, since $d_e = - \sum_{g \neq e} d_g$ and
	by~\eqref{eqn:trace_of_dual_z}:
	\[
	\sum_{g \in G} (dz^{\vee})_g = \frac{\tr_{X/Y}(dz^{\vee})}{\tr_{X/Y}(z)} 
	= \frac{d \tr_{X/Y}(z^{\vee})}{\tr_{X/Y}(z)} = 0.
	\]
	Using multiple times Replacement Theorem (cf. \cite[Theorem 11.3]{DummitFoote2004}) one can construct a basis $\mc B$ of the $k$-vector space $J'$ such that:
	\[
		V = \Span_k(\mc B \cap V) \quad \textrm{ for } V \in \{ E_Q : Q \in B \} \cup \{ E_Q \cap \mc D : Q \in B \} \cup \{ \mc D \}.
	\]
	We define the $k$-linear homomorphism $\varphi : J' \to \Omega_{k(Y)}$ by its values on $\mc B$:
	\begin{itemize}
		\item for $b \in \mc B \cap \mc D$ we define $\varphi(b) = \varphi_{\mc D}(b)$,
		\item for any $b \in \mc B \setminus \mc D$, $\varphi(b) \in \Omega_{k(Y)}$ is any differential form
		that is regular on $Y \setminus B$ and satisfies:
		\begin{align*}
			(\varphi(b))_{Q \in B} &= \psi(b) \textrm{ in } \bigoplus_{Q \in B} \Omega_{k(Y)}/\Omega_Y(B)_Q,\\
			\res_{Q}(\varphi(b)) &= 0 \quad \textrm{ for every } Q \in B
		\end{align*}
		(observe that such a form exists by~\eqref{eqn:constructing_diff_forms}).
	\end{itemize}
	Then $\varphi(\alpha) = \varphi_{\mc D}(\alpha)$ for every $\alpha \in \mc D$, since this holds on the basis $\mc D \cap \mc B$ of $\mc D$.
	Moreover, for every $\alpha \in J'$:
	\begin{equation} \label{eqn:phi_equals_psi}
		(\varphi(\alpha))_{Q \in B} = \psi(\alpha) \textrm{ in } \bigoplus_{Q \in B} \Omega_{k(Y)}/\Omega_Y(B)_Q.
	\end{equation}
	Indeed, for every $g \in G$:
	\[
		(\varphi(d_g))_{Q \in B} = ((d z^{\vee})_g)_{Q \in B} = \sum_{Q \in B} (d z^{\vee})_g \cdot \delta(Q)
		= \psi(d_g) \textrm{ in } \bigoplus_{Q \in B} \Omega_{k(Y)}/\Omega_Y(B)_Q.
	\]
	Hence, \eqref{eqn:phi_equals_psi} holds for $\alpha \in \mc D$. This in turn implies that~\eqref{eqn:phi_equals_psi} holds on basis $\mc B$.
	Therefore it holds for every $\alpha \in J'$.
	Observe also that:
	\begin{equation} \label{eqn:phi_EQ}
		\res_Q(\varphi(\alpha)) = 0 \qquad \textrm{ for every } Q \in B \textrm{ and } \alpha \in E_Q.
	\end{equation}
	Indeed, suppose that $\alpha = \sum_{b \in \mc B} a_b \cdot b \in E_Q$ for $a_b \in k$. Then, by the definition
	of $\varphi$ and~$\mc B$:
	\[
		\res_Q(\varphi(\alpha)) = \sum_{b \in \mc B} a_b \cdot \res_Q(\varphi(b)) = \sum_{b \in \mc B \cap E_Q \cap \mc D} a_b \cdot \res_Q(\varphi(b)).
	\]
	But $\varphi(E_Q \cap \mc D) \subset \Omega_{Y, Q}$, since $E_Q \cap \mc D = \Span_k ( \sum_{g \in g_0 G_Q} d_g : g_0 \in G )$ by Step II 
	and for any $g_0 \in G$:
	\[
		\varphi \left(\sum_{g \in g_0 G_Q} d_g \right) = \sum_{g \in g_0 G_Q} (dz^{\vee})_g \in \Omega_{Y, Q}
	\]
	(see~\eqref{eqn:chi_is_regular} for the last inclusion). Thus $\res_Q(\varphi(\alpha)) = 0$.\\
	
	\noindent \bb{Step IV: we define $\xi_Q$.}\\
	\emph{Proof of Step IV:} Define for any $Q \in B$:
	\[
	\xi_Q := \sum_{g \in G} g^*(z) \cdot \varphi(\ol g_Q).
	\]
	We check now that the forms~$\xi_Q$ satisfy the listed conditions. 
	Note that for any $g \in G$:
	\[
	(dz^{\vee})_g = \varphi(d_g) = \varphi \left(\sum_{Q \in B} \ol g_Q \right) = \sum_{Q \in B} \xi_{Q, g}.
	\]
	Hence:
	\begin{align*}
		\sum_{Q \in B} \xi_Q &= \sum_{Q, g} g^*(z) \xi_{Q, g}
		= \sum_{g \in G} g^*(z) \sum_{Q \in B} \xi_{Q, g}\\
		&= \sum_{g \in G} g^*(z) (dz^{\vee})_g = dz^{\vee},
	\end{align*}
	which proves~(4). Moreover for any $g_0 \in G$, $Q \in B$:
	\begin{align*}
		0 &= \varphi \left(\sum_{g \in g_0 G_Q} \ol g_Q \right) = \sum_{g \in g_0 G_Q} \xi_{Q, g}.
	\end{align*}
	Hence, if $G/G_Q = \{ g_1 G_Q, \ldots, g_r G_Q \}$, using the normality of $G_Q$ in $G$
	and the equality~\eqref{eqn:tr_XQ_gz}:
	\begin{align*}
		\sum_{g \in g_0 G_Q} g^*(\xi_Q) &= \sum_{g \in G_Q g_0} g^*(\xi_Q)
		= g_0^*(\tr_{X/X_Q}(\xi_Q))\\
		&= g_0^* \left(\sum_{i = 1}^r g_i^*(z_Q) \sum_{g \in g_i G_Q} \xi_{Q, g} \right)\\
		&= \sum_{i = 1}^r (g_i \cdot g_0)^*(z_Q) \sum_{g \in g_i G_Q} \xi_{Q, g} = 0
	\end{align*}
	and~(3) is also true.\\
	We prove now the properties (1) and~(2). Using~\eqref{eqn:phi_equals_psi}, we have for every $Q \in B$, $g \in G$:
	\begin{align*}
		\varphi(\ol g_Q) &\equiv (dz^{\vee})_g \pmod{\Omega_Y(B)_Q}\\
		\varphi(\ol g_Q) &\equiv 0 \pmod{\Omega_Y(B)_{Q'}} \qquad \textrm{ for } Q' \in B, Q' \neq Q.
	\end{align*}
	Using Lemma~\ref{lem:inclusions_of_modules} one easily deduces that $\xi_Q \in \Omega_X(R)_{Q'}$
	for $Q' \in B$, $Q' \neq Q$ and $\xi_Q - dz^{\vee} \in \Omega_X(R)_Q$. Moreover, for any
	$g_0 \in G$ and $Q' \in B$, $Q' \neq Q$, using~\eqref{eqn:phi_EQ}:
	\begin{align*}
		\sum_{g \in g_0 G_{Q'}} \res_{Q'}(\xi_{Q, g}) &= 
		\res_{Q'} \left(\sum_{g \in g_0 G_{Q'}} \varphi(\ol g_Q) \right)\\
		&= \res_{Q'}\left(\varphi\left( \sum_{g \in g_0 G_{Q'}} \ol g_Q \right) \right) = 0,
	\end{align*}
	since $\sum_{g \in g_0 G_{Q'}} \ol g_Q \in E_{Q'}$.
	Therefore (2) holds by Lemma~\ref{lem:main_lemma_Omega_Y}~(2). Analogously,
	for any $Q \in B$, $g_0 \in G$, using~\eqref{eqn:chi_is_regular}:
	\begin{align*}
		\sum_{g \in g_0 G_{Q}} \res_Q(\xi_{Q, g} - (dz^{\vee})_g)
		&= \res_Q \left(\sum_{g \in g_0 G_{Q}} \varphi(\ol g_Q) \right) - \res_Q \left(\sum_{g \in g_0 G_Q} (dz^{\vee})_g \right)\\
		&= \res_Q \left(\varphi \left(\sum_{g \in g_0 G_{Q}} \ol g_Q \right) \right) - \res_Q \left(\sum_{g \in g_0 G_Q} (dz^{\vee})_g \right)\\
		&= \res_Q (0) - 0 = 0.
	\end{align*}
	Thus $\xi_Q - dz^{\vee} \in \Omega_{X, Q}$ by Lemma~\ref{lem:main_lemma_Omega_Y}~(2), which proves~(1).
\end{proof}
We need one more technical result before the proof of Proposition~\ref{prop:extending_resG}.
\begin{Lemma} \label{lem:cartesian}
	Suppose that in an abelian category we have the following commutative diagram: 
	\vspace{-1em}
	%
		\begin{equation} \label{eqn:two_squares_diagram}
			\begin{tikzcd}
				A \arrow[r, hook] \arrow[d, hook] & B \arrow[r, two heads] \arrow[d, hook] & C \arrow[d, hook] \\
				A' \arrow[r, hook]                & B' \arrow[r, two heads]                & C'.              
			\end{tikzcd}
		\end{equation}
	Assume also that:
	\begin{itemize}
		\item the map $B \to B'$ induces an isomorphism $\coker(A \to B) \cong \coker(A' \to B')$,
		
		\item the maps $A \to C$ and $A' \to C'$ are surjective,
		
		\item the outer rectangle of~\eqref{eqn:two_squares_diagram} is a cartesian diagram.
	\end{itemize}
	Then the right square is a cartesian diagram.
\end{Lemma}
\begin{proof}
	Note that the outer rectangle of~\eqref{eqn:two_squares_diagram} is cocartesian by \cite[\href{https://stacks.math.columbia.edu/tag/08N4}{Lemma 08N4 (1)}]{stacks-project}. Moreover, the condition $\coker(A \to B) \cong \coker(A' \to B')$ implies that the left square of~\eqref{eqn:two_squares_diagram} is cocartesian by~\cite[Prop. 2.12 (i) $\Leftrightarrow$ (iv)]{Buhler_Exact_Categories}. Therefore the right square is also cocartesian,
	cf. the dual version of \cite[ex. III.4.8 (b)]{MacLane_Categories_working}. But this implies that it is also cartesian by \cite[\href{https://stacks.math.columbia.edu/tag/08N4}{Lemma 08N4~(2)}]{stacks-project}.
\end{proof}

\begin{proof}[Proof of Proposition~\ref{prop:extending_resG}] \mbox{}\\
	(1) Define:
	\begin{align*}
		\res_G : Z^1_B(\pi_* \Omega_X^{\bullet}) &\to \bigoplus_{Q \in B} k[G],\\
		(\omega, \nu) &\mapsto \sum_{Q, g} (\res_Q((\omega - d \nu_Q)_g) - \langle g^*(\xi_Q), \nu \rangle) \cdot g_Q,
	\end{align*}
	where $\xi_Q$ are as in Lemma~\ref{lem:xiQ_existence}, $\langle \cdot, \cdot \rangle$ is the duality pairing~\eqref{eqn:duality_pairing_formula}, $\omega \in \Omega_X(V)$, $\nu \in \bigoplus_{Q \in B} k(X)$. We'll show that $\res_G$ induces the desired homomorphisms $H^1_{dR}(X) \to I_{X/Y}$ and $H^1_{dR, log}(X, R) \to k[G]_B$.
	We start by checking that the image of this map lies in $I_{X/Y}$.
	Indeed, $\omega - d \nu_Q \in \Omega_{X, Q}$ implies that
	\[
	\sum_{g \in G} \res_Q((\omega - d \nu_Q)_g) \cdot g \in I_{G, G_Q}
	\]
	by Lemma~\ref{lem:main_lemma_Omega_Y}.
	Moreover, by~Lemma~\ref{lem:xiQ_existence}~(3):
	\[
	\sum_{g \in G} \langle g^*(\xi_Q), \nu \rangle \cdot g \in I_{G, G_Q}.
	\]
	Using Lemma~\ref{lem:xiQ_existence}~(4), \eqref{eqn:residue_and_trace}, the equality $\res_P(f \cdot dh) = - \res_P(h \cdot df)$ for any $f, h \in k(X)$ and~\eqref{eqn:gth_component_trace_f} we obtain:
	\begin{align*}
		\sum_{Q \in B} \langle g^*(\xi_Q), \nu \rangle &=
		\langle dg^*(z^{\vee}), \nu \rangle
		= \sum_{P \in R} \res_P(\nu_{\pi(P)} \cdot dg^*(z^{\vee}))\\
		&= - \sum_{P \in R} \res_P(g^*(z^{\vee}) \cdot d\nu_{\pi(P)})
		= - \sum_{Q \in B} \res_Q((d\nu_Q)_g)
	\end{align*}
	for any $g \in G$. Hence by the residue theorem:
	\begin{align*}
		\sum_{Q, g} (\res_Q((\omega - d \nu_Q)_g) - \langle g^*(\xi_Q), \nu \rangle) \cdot g &= 
		\sum_{Q, g} \res_Q(\omega_g) \cdot g = 0.	
	\end{align*}
	We check now that this map vanishes on $B^1_B(\pi_* \Omega_X^{\bullet})$. Suppose that $h \in \mc O_X(V)$. Then $(dh, (h)_{Q \in B})$
	maps to:
	\[
	\sum_{Q, g} \left(\res_Q((dh - dh)_g) - \langle g^*(\xi_Q), (h)_{Q \in B} \rangle \right) \cdot g_Q = 0.
	\]
	Moreover, if $\nu \in \bigoplus_{Q \in B} \mc O_{X, Q}$ then by \eqref{eqn:residue_and_trace}, \eqref{eqn:gth_component_trace_f}
	and by Lemma~\ref{lem:xiQ_existence} (1), (2):
	\begin{align*}
		\langle g^*(\xi_Q), \nu \rangle &= \sum_{Q' \in B} \sum_{P \in \pi^{-1}(Q')} \res_P(\nu_{Q'} \cdot g^*(\xi_Q))\\
		&= \sum_{P \in \pi^{-1}(Q)} \res_P(\nu_Q \cdot dg^*(z^{\vee}))\\
		&= -\sum_{P \in \pi^{-1}(Q)} \res_P(g^*(z^{\vee}) \cdot d\nu_Q)\\
		&= -\res_Q(\tr_{X/Y}(g^*(z^{\vee}) \cdot d\nu_Q))\\
		&= -\res_Q((d\nu_Q)_g).
	\end{align*}
	Hence:
	\begin{equation*}
		\res_G((0, \nu)) = \sum_{Q, g} \left(\res_Q((0 - d \nu_Q)_g) - \langle g^*(\xi_Q), \nu \rangle \right) \cdot g_Q
		= 0.
	\end{equation*}
	This proves that $\res_G : H^0(X, \Omega_X) \to I_{X/Y}$ extends to $\res_G : H^1_{dR}(X) \to I_{X/Y}$.
	This map is a split surjection, since $\res_G : H^0(X, \Omega_X) \to I_{X/Y}$ has a section.
	Similarly one obtains a split surjection $\res_G : H^1_{dR, log}(X, R) \to k[G]_B$.
	Note that the assumptions of Lemma~\ref{lem:cartesian} with $A := H^0(X, \Omega_X)$, $B := H^1_{dR}(X)$, $C := I_{X/Y}$, $A' := H^0(X, \Omega_X(R))$, $B' := H^1_{dR, log}(X, R)$, $C' := k[G]_B$ are satisfied. Indeed, this follows by the Hodge--de Rham sequences~\eqref{eqn:intro_hodge_de_rham_se} and~\eqref{eqn:hdr_exact_sequence_for_log_de_rham}, and by the proof of Proposition~\ref{prop:log_diffs_and_diffs}. Hence the diagram:
	\begin{center}
		\begin{tikzcd}
			H^1_{dR}(X) \arrow[r, hook] \arrow[d, two heads] & {H^1_{dR, log}(X, R)} \arrow[d, two heads] \\
			I_{X/Y} \arrow[r, hook]                            & {k[G]_B}                                    
		\end{tikzcd}
	\end{center}
	is cartesian and the maps $\res_G : H^1_{dR}(X) \to I_{X/Y}$ and $\res_G : H^1_{dR, log}(X, R) \to k[G]_B$ have isomorphic kernels by \cite[\href{https://stacks.math.columbia.edu/tag/08N3}{Lemma 08N3~(1)}]{stacks-project}.\\

(2) Observe that for any $Q \in B$ the element
\[
\Xi_Q := (\xi_Q, z^{\vee} \cdot \delta(Q)) \in \Omega_X(V) \times \bigoplus_B k(X)
\]
belongs to $Z^1_B(\pi_* \Omega_X^{\bullet})$ by Lemma~\ref{lem:xiQ_existence}~(1) and~(2). We show now that
the map:
\begin{equation} \label{eqn:pre_lifted_res_vee}
\bigoplus_{Q \in B} k[G] \to Z^1_B(\pi_* \Omega_X^{\bullet}), \qquad g_Q \mapsto g^*(\Xi_Q)
\end{equation}
induces a $k[G]$-linear map $J_{X/Y} \to H^1_{dR, log}(X)$ lifting $\res_G^{\vee} : J_{X/Y} \to H^1(X, \mc O_X)$.
Note that by Lemma~\ref{lem:xiQ_existence}(3) and Corollary~\ref{cor:main_lemma_OX} for any $g_0 \in G$, $Q \in B$:
	\begin{align*}
		\sum_{g \in g_0 G_Q} g^*((\xi_Q, z^{\vee} \cdot \delta(Q))) 
		&= (0, g_0^*(z_Q^{\vee}) \cdot \delta(Q)) \in B^1_B(\pi_* \Omega_X^{\bullet}).
	\end{align*}
Moreover, by Lemma~\ref{lem:xiQ_existence}~(4):
	\begin{align*}
		\sum_{Q \in B} g^*(\Xi_Q) &= (d g^*(z^{\vee}), \sum_{Q \in B} g^*(z^{\vee}) \delta(Q))\\
		&= (d g^*(z^{\vee}), (g^*(z^{\vee}))_{Q \in B}) \in B^1_B(\pi_* \Omega_X^{\bullet}).
	\end{align*}
	Hence in $H^1_{dR, log}(X)$:
	\begin{align*}
		\forall_{g_0 \in G} \, \forall_{Q \in B} \, \sum_{g \in g_0 G_Q} g^*(\Xi_Q) = 0 \quad \textrm{ and } \quad \forall_{g 
		\in G}
		\sum_{Q \in B} g^*(\Xi_Q) = 0,
	\end{align*}
	which proves the claim. Analogously, one shows that the map~\eqref{eqn:pre_lifted_res_vee} induces a map $\res_G^{\vee} : k[G]_B^{\vee} \to H^1_{dR, qlog}(X, R)$
	lifting the map $\res_G^{\vee} : k[G]_B^{\vee} \to H^1(X, \mc O_X(-R))$. Finally, the homomorphisms 
	$\res_G^{\vee} : J_{X/Y} \to H^1_{dR, log}(X, R)$  and  $\res_G^{\vee} : k[G]_B^{\vee} \to H^1_{dR, qlog}(X, R)$ have isomorphic cokernels
	by the dual version of Lemma~\ref{lem:cartesian}.
\end{proof}

\begin{proof}[Proof of Theorem~\ref{thm:main_thm}, part 2]
	Note that by Proposition~\ref{prop:extending_resG} we have:
	\begin{align*}
		H^1_{dR}(X) \oplus k[G]_B &\cong H^1_{dR, log}(X, R) \oplus I_{X/Y},\\
		H^1_{dR, log}(X, R) \oplus k[G]_B^{\vee} &\cong H^1_{dR, qlog}(X, R) \oplus J_{X/Y}.
	\end{align*}
	Hence:
	\begin{equation} \label{eqn:dr_qdr_connection}
		H^1_{dR}(X) \oplus k[G]_B \oplus k[G]_B^{\vee} \cong H^1_{dR, qlog}(X, R) \oplus I_{X/Y} \oplus J_{X/Y}.
	\end{equation}
	We describe now the structure of $H^1_{dR, qlog}(X, R)$. Note that by Riemann--Roch theorem, $\dim_k H^1(Y, \mc O_Y(-B)) = g_Y + \# B - 1$. Let $K_1$ be the kernel of the epimorphism:
	\begin{alignat*}{3}
		H^1_{dR, qlog}(X, R) &\to \quad H^1(X, \mc O_X(-R)) &&\to H^1 \left(Y, \bigoplus_{g \in G} g^*(z^{\vee}) \mc O_Y(-B) \right)&\\
		&\cong \Ind^G_1 H^1(Y, \mc O_Y(-B)) &&\cong k[G]^{g_Y + \# B - 1},&
	\end{alignat*}
	where the first map comes from the Hodge--de Rham exact sequence~\eqref{eqn:hdr_exact_sequence_for_log_de_rham}
	and the second map comes from the inclusion from the Lemma~\ref{lem:inclusions_of_modules2}
	(note that the second map is surjective, since the first cohomology of a finitely supported sheaf on a curve
	vanishes).
	Then, since $k[G]^{g_Y + \# B - 1}$ is a projective $k[G]$-module:
	\begin{equation} \label{eqn:H1_dR=K1+k[G]}
		H^1_{dR, qlog}(X, R) \cong K_1 \oplus k[G]^{g_Y + \# B - 1}.
	\end{equation}
	Moreover, one can explicitly describe $K_1$ as:
	\[
		K_1 = \frac{\{ (\omega, \nu) \in \Omega_X(V) \times K_{1, 1} : \omega - d\nu_Q \in \Omega_X(Q)_Q \quad \forall_{Q \in B} \}}{\{ (df, (f + f_Q)_{Q \in B}) : f \in K_{1, 2}, \quad f_Q \in \mc O_X(-R)_Q \, \forall_{Q \in B} \}},
	\]
	where:
	\begin{align*}
		K_{1, 1} &:= \bigoplus_{Q \in B} \bigoplus_{g \in G} g^*(z^{\vee}) \mc O_Y(-B)_Q,\\
		K_{1, 2} &:= \mc O_X(V) \cap \bigcap_{Q \in B} \bigoplus_{g \in G} g^*(z^{\vee}) \mc O_Y(-B)_Q.
	\end{align*}
	However, by Lemma~\ref{lem:inclusions_of_modules2} and Lemma~\ref{lem:properties_H0Q_H1Q}~(2):
	\begin{align*}
		K_{1, 2} &= \bigoplus_{g \in G} g^*(z^{\vee}) \mc O_Y(U) \cap 
		\bigcap_{Q \in B} \bigoplus_{g \in G} g^*(z^{\vee}) \mc O_Y(-B)_Q\\
		&= \bigoplus_{g \in G} g^*(z^{\vee}) \mc O_Y(-B)(Y) = 0.
	\end{align*}
	Hence by~\eqref{eqn:H1Q}:
	\begin{align*}
		K_1 &= \frac{\{ (\omega, \nu)  \in \Omega_X(V) \times K_{1, 1} : \omega - d\nu_Q \in \Omega_X(Q)_Q \quad \forall_{Q \in B} \}}{\{ (0, (f_Q)_{Q \in B}) : f_Q \in \mc O_X(-R)_Q
			\quad \forall_{Q \in B} \}}\\
		&= \{ (\omega, \nu) \in \Omega_X(V) \times \bigoplus_{Q \in B} H^1_Q : \omega - d\nu_Q \in \Omega_X(Q)_Q \quad \forall_{Q \in B} \}.
	\end{align*}	
	The Hodge--de Rham exact sequence~\eqref{eqn:hdr_exact_sequence_for_qlog_de_rham} implies that the image of the composition
	\[
		H^0 \left(Y, \bigoplus_{g \in G} g^*(z) \Omega_Y(B) \right) \to H^0(X, \Omega_X(R)) \to H^1_{dR, qlog}(X, R)
	\]
	(where the first map comes from Lemma~\ref{lem:inclusions_of_modules} and the second one from~\eqref{eqn:hdr_exact_sequence_for_qlog_de_rham}) lands in $K_1$. Therefore we obtain a monomorphism:
	\begin{equation} \label{eqn:monomorphism_K1}
		H^0 \left(Y, \bigoplus_{g \in G} g^*(z) \Omega_Y(B) \right) \to K_1.
	\end{equation}
	Let $K_2$ be the cokernel of the monomorphism~\eqref{eqn:monomorphism_K1}. Then, analogously as before, using Riemann--Roch theorem and the injectivity of $k[G]$:
	\begin{align}
		K_1 \cong H^0 \left(Y, \bigoplus_{g \in G} g^*(z) \Omega_Y(B) \right) \oplus K_2 \cong k[G]^{g_Y + \# B - 1} \oplus K_2. \label{eqn:K1=K2+k[G]}
	\end{align}
	On the other hand:
	\begin{align*}
		K_2 &= \frac{\{ (\omega, \nu) \in \Omega_X(V) \times \bigoplus_{Q \in B} H^1_Q: \omega - d\nu_Q \in \Omega_X(Q)_Q \}}{\{ (\omega, 0) : \omega \in \bigoplus_{g \in G} g^*(z) \Omega_Y(B)(Y) \}}\\
		&\cong \left\{ (\omega, \nu) \in 
			\frac{\Omega_X(V)}{\bigoplus_{g \in G} g^*(z) \Omega_Y(B)(Y)} \times \bigoplus_{Q \in B} H^1_Q : \omega - d\nu_Q \in \Omega_X(Q)_Q \right\}.
	\end{align*}
	Let $j : V \to X$ be the open immersion.
	Note that both sheaves:
	\[
		\frac{\pi_* \Omega_X(R)}{\bigoplus_{g \in G} g^*(z) \Omega_Y(B)} \textrm{ and } 
		\frac{(\pi \circ j)_* \, \Omega_{V}}{\pi_* \Omega_X(R)}
	\]
	have support contained in $B$. Hence $(\pi \circ j)_* \Omega_{V}/\bigoplus_{g \in G} g^*(z) \Omega_Y(B)$ is also supported on~$B$.
	This implies that:
	\[
		\frac{\Omega_X(V)}{\bigoplus_{g \in G} g^*(z) \Omega_Y(B)(Y)} \cong \bigoplus_{Q \in B}
		\frac{\Omega_{k(X)}}{\bigoplus_{g \in G} g^*(z) \Omega_Y(B)_Q}.
	\]
	Therefore:
	\[
		K_2 \cong \bigoplus_{Q \in B} H^1_{dR, Q},
	\]
	where:
	\begin{align}
		H^1_{dR, Q} &:= \{ (\omega_Q, \nu_Q) \in \frac{\Omega_{k(X)}}{\bigoplus_{g \in G} g^*(z) \Omega_Y(B)_Q} \times H^1_Q :
		\quad \omega_Q - d\nu_Q \in \Omega_X(R)_Q \} \label{eqn:H1dRQ}
	\end{align}
	(note that $\omega_Q - d\nu_Q$ is well-defined only modulo $\bigoplus_{g \in G} g^*(z) \Omega_Y(B)_Q + d \mc O_X(-R)_Q \subset \Omega_X(R)_Q$).
	Thus by~\eqref{eqn:H1_dR=K1+k[G]} and~\eqref{eqn:K1=K2+k[G]}:
	\begin{align*}
		H^1_{dR, qlog}(X, R) &\cong K_1 \oplus k[G]^{g_Y + \# B - 1}\\
		&\cong K_2 \oplus k[G]^{2 \cdot (g_Y + \# B - 1)}\\
		&\cong \bigoplus_{Q \in B} H^1_{dR, Q} \oplus k[G]^{2 \cdot (g_Y + \# B - 1)}.
	\end{align*}
	One finishes the proof using~\eqref{eqn:dr_qdr_connection}.
\end{proof}
In the sequel we will need to know the relation between $H^0_Q$, $H^1_Q$, $H^1_{dR, Q}$. This is provided
by the following local analogue of the Hodge--de Rham exact sequence.
\begin{Lemma} \label{lem:properties_IXY_H1dRQ}
	For every $Q \in B$ there exists an exact sequence:
	\begin{equation*}
		0 \to H^0_Q \to H^1_{dR, Q} \to H^1_Q \to 0.
	\end{equation*}
\end{Lemma}
\begin{proof}
	It is straightforward that the map $H^0_Q \to H^1_{dR, Q}$ induced
	by $\omega_Q \mapsto (\omega_Q, 0)$ is injective. The map $H^1_{dR, Q} \to H^1_Q$, $(\omega_Q, \nu_Q) \mapsto \nu_Q$ is surjective, since $\nu_Q$ is the image of $(d \nu_Q, \nu_Q)$. Finally, $(\omega_Q, \nu_Q)$ is
	in the kernel of the map $H^1_{dR, Q} \to H^1_Q$ if and only if
	$\nu_Q \in \mc O_X(-R)_Q$, which is equivalent to $\omega_Q \in \Omega_X(R)$ (since $\omega_Q - d \nu_Q \in \Omega_X(R)_Q$).
\end{proof}
We prove now the corollary concerning the structure of $H^1_{dR}(X)$ in the weak ramification case.
\begin{proof}[Proof of Corollary~\ref{cor:hdr_exact_sequence}]
	The proof of~\cite[Main Theorem]{Garnek_equivariant} shows that if $H^1_{dR}(X) \cong H^0(X, \Omega_X) \oplus H^1(X, \mc O_X)$ as $k[G]$-modules, then $d_P'' = 0$ for every $P \in X(k)$.
	Suppose now that $d_P'' = 0$ for every $P \in X(k)$. Then
	$H^0_Q = H^1_Q = H^1_{dR, Q} = 0$ for every $Q \in Y(k)$ by Lemma~\ref{lem:properties_H0Q_H1Q} and Lemma~\ref{lem:properties_IXY_H1dRQ}.
	Hence by Theorem~\ref{thm:main_thm}:
	\[
		H^1_{dR}(X) \cong k[G]^{\oplus 2 \cdot g_Y} \oplus I_{X/Y} \oplus J_{X/Y}
		\cong H^0(X, \Omega_X) \oplus H^1(X, \mc O_X).
	\]
	This implies that the Hodge--de Rham exact sequence splits by \cite[Theorem 3.5 and first paragraph of \S 4]{Guralnick_Roths_theorems}, since we can identify $H^0(X, \Omega_X)$ with a submodule of $H^1_{dR}(X)$
	using the Hodge--de Rham exact sequence~\eqref{eqn:intro_hodge_de_rham_se}.
\end{proof}

\section{Artin--Schreier covers} \label{sec:AS_covers}
In this subsection we construct a magical element for a large class of Artin--Schreier covers
and compute their cohomology. Also, we give an example of a $\ZZ/p$-cover without a magical element.\\

Let $k$ be an algebraically closed field of characteristic $p$ and $Y/k$ be a smooth projective curve.
Recall that the $\ZZ/p$-covers
of $Y$ are in a bijection with the group $k(Y)/\wp(k(Y))$, where $\wp(f) := f^p - f$. An element of
$k(Y)/\wp(k(Y))$ represented by a function $f \in k(Y)$ corresponds to a curve $X$
with the function field given by the equation:
\begin{equation} \label{eqn:artin_schreier}
	y^p - y = f.
\end{equation}
Note that $X$ is irreducible, if and only if $f \not \in \wp(k(Y))$.
The action of $G = \langle \sigma \rangle \cong \ZZ/p$ on $X$ is 
then given by $\sigma(y) := y+1$.  From now on we assume that $\pi : X \to Y$ is a $\ZZ/p$-cover given by an equation of the form~\eqref{eqn:artin_schreier}. We say that $y$ is an \bb{Artin--Schreier generator} of $\pi$.
\subsection{Local standard form}
The Artin--Schreier generator $y$ is in \bb{local standard form at $Q \in Y(k)$},
if it satisfies an equation of the form~\eqref{eqn:artin_schreier}, in which either $f$ is regular at~$Q$ or $f$ has a pole of order not divisible by $p$ at $Q$.
Note that for every $Q \in Y(k)$, there exists an Artin--Schreier generator of $X$ in local standard form at~$Q$. Indeed, given an arbitrary equation of the form~\eqref{eqn:artin_schreier} one can
repeatedly replace $y$ by $y - g$ and $f$ by $f - \wp(g)$, where $g$ is a power
of a suitably chosen uniformizer at~$Q$.\\

\noindent Denote $m_Q = \max \{ i : G_{Q, i} \neq 0 \}$.
Then $d_Q = (m_Q + 1) \cdot (p-1)$. It turns out that if $y$ is in local
standard form at $Q$ then:
\begin{equation} \label{eqn:mQ_formula}
	m_Q = m_{X/Y, Q} := 
	\begin{cases}
		|\ord_Q(f)|, & \textrm{ if } \ord_Q(f) < 0,\\
		0, & \textrm{ otherwise}
	\end{cases}
\end{equation}
(see e.g.~\cite[Lemma 4.2]{Garnek_equivariant}). 
The following result allows us to compute $m_Q$ without finding a local standard form in most cases.
\begin{Lemma} \label{lem:computing_mQ} 
	Let $y$ be an arbitrary Artin--Schreier generator of a $\ZZ/p$-cover of $Y$
	given by~\eqref{eqn:artin_schreier}. Suppose that $Q \in Y(k)$ is a pole of $f$
	satisfying:
	\begin{equation} \label{eqn:assumption_to_compute_mQ}
		\ord_Q(df) < \frac{1}{p} \ord_Q(f) - 1.
	\end{equation}
	Then $m_Q = -\ord_Q(df) - 1$.
\end{Lemma}
\begin{proof}
	Note that there exist $g, h \in k(Y)$ such that $f = \wp(g) + h$, $\ord_Q(g) = \ord_Q(f)/p$ and:
	\[ \ord_Q(h) = 
	\begin{cases}
		-m_Q, & \textrm{ if } m_Q \ge 1,\\
		\ge 0, & \textrm{ if } m_Q = 0.
	\end{cases}
	\]
	Then $df + dg = dh$. Observe that $\ord_Q(dg) \ge \ord_Q(g) - 1 = \frac{1}{p} \ord_Q(f) - 1 > \ord_Q(df)$
	and hence:
	\[
	-m_Q - 1 = \ord_Q(dh) = \min(\ord_Q(df), \ord_Q(dg)) =  \ord_Q(df). \qedhere
	\]
\end{proof}

\subsection{Global standard form} \label{subsec:gsf}
The Artin--Schreier generator $y$ is said to be in \bb{global standard form}, if it
is in local standard form at every $Q \in Y(k)$ and $y^p - y \not \in k$ (i.e. $\pi$ is not the trivial disconnected cover $\bigsqcup_{\ZZ/p} Y$). In this situation we say also that $\pi$ has a global standard form. The following result explains our interest in this notion.
\begin{Lemma} \label{lem:gsf_gives_me_for_Zp}
	Suppose that $y$ is an Artin--Schreier generator in global standard form for~$\pi$.
	Then $z := y^{p-1}$ is a magical element for
	$\pi$ and the dual of $z$ with respect to the trace pairing is $z^{\vee} := z - 2$.
\end{Lemma}
\begin{proof}
	Note that $\ord_P(z) = (p-1) \cdot \ord_{P}(y) = - (p-1) \cdot m_P = -d_P'$
	for every $P \in R$. Moreover, by \cite[\S 3, Proposition 3]{Madden_Arithmetic_generalized} for $i \le 2 \cdot (p-1)$ we have:
	\begin{equation} \label{eqn:AS_trace_of_yi}
			\tr_{X/Y}(y^i) =
		\begin{cases}
			0, & \textrm{ if } i = 0 \textrm{ or } (p-1) \nmid i.\\
			-1, & \textrm{ else.}
		\end{cases}
	\end{equation}
	This allows us to conclude that $\tr_{X/Y}(z) \neq 0$ and $z^{\vee}$ is the dual element to~$z$.
\end{proof}

Not every $\ZZ/p$-cover has a global standard form. For example,
connected \'{e}tale $\ZZ/p$-covers do not have a global standard form (cf.~\cite[Subsection~3.1]{WardMarques_HoloDiffs}). We present another example in Subsection~\ref{subsec:no_magical_element}.
It turns out that every sufficiently ramified $\ZZ/p$-cover has a global standard form.
\begin{Lemma} \label{lem:criterion_for_gsf}
	Suppose that there exists a point $Q_0 \in Y(k)$ with $m_{Q_0} > 2g_Y \cdot p$. Then the cover $\pi$
	has a global standard form.
\end{Lemma}
\begin{proof} 
	Let $y$ be an Artin--Schreier generator for $\pi$ in local standard form at $Q_0$ and let $y^p - y = f$.
	Suppose that $\ord_Q(f) = -j < 0$, $p|j$ for some $Q \in Y(k)$. Let $D := 2 g_Y \cdot Q_0 + \frac{j}{p} \cdot Q \in \Divv(Y)$. Using Riemann--Roch theorem, one concludes that
	\[
	H^0(Y, \mc O_Y(D)) \setminus \bigg( H^0(Y, \mc O_Y(D - Q_0)) \cup H^0(Y, \mc O_Y(D - Q)) \bigg) \neq \varnothing,
	\]
	(note that a vector space over an infinite field cannot be a union of two codimension~$1$ subspaces). Therefore we may choose a function
	$g \in k(Y)$ such that:
	\begin{itemize}
		\item $\ord_{Q_0}(g) = -2g_Y$,
		\item $\ord_Q(g) = -j/p$,
		\item $\ord_{Q'}(g) \ge 0$ for $Q' \neq Q, Q_0$.
	\end{itemize}
	Note that there exists $c \in k$ such that $\ord_Q(f - \wp(c \cdot g)) > -j$. Let $y_1 := y - c \cdot g$, $f_1 := f - \wp(c \cdot g)$. Then $y_1$ is an Artin--Schreier generator
	of $\pi$. Moreover:
	\begin{itemize}
		\item $y_1$ is in local standard form at $Q_0$ -- indeed, $\ord_{Q_0}(f_1) = \ord_{Q_0}(f)$, since $\ord_{Q_0}(\wp(g)) = -2g_Y \cdot p > - m_{Q_0}$,
		
		\item if $y$ is in local standard form at $Q' \neq Q, Q_0$,
		then $y_1$ is also in local standard form at $Q'$,
		
		\item $\ord_Q(f_1) > \ord_Q(f)$.
	\end{itemize}
	Thus, by repeatedly replacing $(y, f)$ by $(y_1, f_1)$, we eventually obtain an Artin--Schreier generator of $\pi$ in global standard form.
\end{proof}
\subsection{Proof of Corollary~\ref{cor:cohomology_of_Zp}} \label{sec:Zp}
In this subsection we apply Theorem~\ref{thm:main_thm} to prove Corollary~\ref{cor:cohomology_of_Zp}.
Assume that $y$ is in global standard form.
By Lemma~\ref{lem:gsf_gives_me_for_Zp}, $z := y^{p-1}$ is a magical element for $\pi$ and we may take $z^{\vee} = z - 2$. Recall that every finitely dimensional indecomposable $k[G]$-module is of the form $J_i = k[x]/(x - 1)^i$ for some $i = 1, \ldots, p$, where $\sigma$ acts by multiplication by $x$ (cf. \cite[Theorem 12.1.5]{DummitFoote2004}).
\begin{Lemma} \label{lem:Ji_and_Span_yi}
	For any $0 \le i \le p-1$ the $k[G]$-module $\Span_k(1, y, \ldots, y^i)$ is generated by $y^i$ and is isomorphic to $J_{i+1}$.
\end{Lemma}
\begin{proof}
	For every $j = 0, \ldots, i$ we have: $\sigma^j(y^i) = \sum_{l = 0}^i {i \choose l} j^{i-l} \cdot y^l$. We treat this as a system of equations with $y^l$ as unknowns. Using Vandermonte's determinant we have:
	\[
	\det \left[{i \choose l} j^{i-l} \right]_{0 \le j, l \le i} = \prod_{l = 0}^i {i \choose l} \cdot \prod_{0 \le j_1 < j_2 \le i} (j_2 - j_1) \not \equiv 0 \pmod p.
	\]
	This easily implies that $y^l \in \Span_k(\sigma^j(y^i) : j = 0, \ldots, i)$
	for $l = 0, \ldots i$. Thus $\Span_k(1, y, \ldots, y^i)$ is generated as a $k[G]$-module by
	$y^i$ and is an irreducible $k[G]$-module of dimension $i+1$.
\end{proof}
\noindent Lemma~\ref{lem:Ji_and_Span_yi} implies in particular that:
\begin{equation} \label{eqn:span_gz}
	\Span_k(1, y, \ldots, y^{p-1}) = \Span_k(g^*(z) : g \in G) = \Span_k(g^*(z^{\vee}) : g \in G)
\end{equation}
(for the second equality note that $\tr(z) = \tr(z^{\vee}) = - 1$ yields $z^{\vee} = z-2 \in \Span_k(g^*(z) : g \in G)$ and $z \in \Span_k(g^*(z^{\vee}) : g \in G)$).

Fix $Q \in B$ and denote $m := m_Q$, $\pi^{-1}(Q) = \{ P \}$.
In order to prove Corollary~\ref{cor:cohomology_of_Zp}, it suffices to compute bases of the modules $H^0_Q$, $H^1_Q$, $H^1_{dR, Q}$, defined by~\eqref{eqn:H0Q}, \eqref{eqn:H1Q} and~\eqref{eqn:H1dRQ} respectively.
To this end, we need to pick an appropriate uniformizer at~$Q$. Using the Hensel's lemma for the polynomial $T^m - \frac 1 f \in \mc O_{Y, Q}[T]$, we can choose $t \in k(Y)$ such that:
\begin{equation} \label{eqn:suitable_uniformizer}
	\frac{1}{t^m} \equiv y^p - y \pmod{\mf m_{Y, Q}^{2m}}.
\end{equation}
\begin{Proposition} \label{prop:H0Q_for_Zp}
	Keep the above setup.
	\begin{enumerate}[(1)]
		\item The elements:
		\[
		y^i \frac{dt}{t^j} \in \Omega_{k(X)},
		\]
		where $0 \le i \le p-2$, $2 \le j$, $mi + pj \le (m+1) \cdot (p-1) + 1$, belong to $\Omega_X(R)_Q$. Their images form a basis of $H^0_Q$.
		\item The elements:
		\[
		y^{i} \cdot t^{j} \in k(X), 
		\]
		where $1 \le i \le p-1$, $1 \le j$, $-mi + pj < 1$, belong to $\bigoplus_{g \in G} g^*(z^{\vee}) \mc O_X(-R)_Q$. Their images form a basis of $H^1_Q$.
		\item The elements:
		\[
		\left(\frac{y^i \, dt}{t^j}, \frac{1}{(i+1) \cdot m} y^{i+1} t^{m+1 - j} \right) \in \Omega_{k(X)} \times k(X),
		\]
		where $0 \le i \le p-2$, $2 \le j \le m$, belong to $\Omega_{k(X)} \times \bigoplus_{g \in G} g^*(z^{\vee}) \mc O_X(-R)_Q$. Their images form a basis of $H^1_{dR, Q}$.
	\end{enumerate}
\end{Proposition}
Before the proof, note that for any $\omega \in \Omega_{k(X)}$ there exists a
unique system of forms $\omega_0, \ldots, \omega_{p-1} \in \Omega_{k(Y)}$ such that
$\omega = \omega_0 + y \cdot \omega_1 + \ldots + y^{p-1} \cdot \omega_{p-1}$.
Moreover, one has:
\begin{equation} \label{eqn:valuation_of_omega_Zp}
	\ord_P(\omega) = \min \{ \ord_P(y^i \omega_i) : i = 0, \ldots, p-1 \}.
\end{equation}
Indeed, using~\eqref{eqn:valuation_of_diff_form} we see that for any $0 \le i < j \le p-1$:
\begin{align*}
	\ord_P(y^i \cdot \omega_i) - \ord_P(y^j \cdot \omega_j) &= (-i m + p \cdot \ord_Q(\omega_i) + d_Q) - (-j m + p \cdot \ord_Q(\omega_j) + d_Q)\\
	&= m \cdot(j - i) + p \cdot (\ord_Q(\omega_i) - \ord_Q(\omega_j)) \neq 0,
\end{align*}
since $p \nmid m \cdot (i - j)$. Analogously, any $f \in k(X)$ is of the form
$f = f_0 + y \cdot f_1 + \ldots + y^{p-1} \cdot f_{p-1}$
for some $f_0, \ldots, f_{p-1} \in k(Y)$ and
\begin{equation} \label{eqn:valuation_of_f_Zp}
	\ord_P(f) = \min \{ \ord_P(y^i f_i) : i = 0, \ldots, p-1 \}.
\end{equation}
\begin{proof}[Proof of Proposition~\ref{prop:H0Q_for_Zp}]
	(1) For any $i$, let $j_{max}(i)$ denote the largest integer satisfying
		$mi + pj \le (m+1) \cdot (p-1) + 1$, i.e. $j_{max}(i) = m+1 - \lceil \frac{m (i+1)}{p} \rceil$. By~\eqref{eqn:valuation_of_omega_Zp} and~\eqref{eqn:valuation_of_diff_form}:
		\begin{align*}
			\omega \in \Omega_X(R)_Q &\Leftrightarrow
			\ord_P(y^i \cdot \omega_i) \ge -1 \textrm{ for every } i = 0, \ldots, p-1\\
			&\Leftrightarrow \ord_Q(\omega_i) \ge -j_{max}(i).
		\end{align*}
		Hence, using~\eqref{eqn:span_gz}:
		\begin{equation*}
			H^0_Q = \frac{\Omega_X(R)_Q}{\left(\bigoplus_{g \in G} g^*(z) k \right) \Omega_Y(B)_Q} = \frac{\bigoplus_{i = 0}^{p-1} y^i \cdot t^{-j_{max}(i)} \Omega_{Y,Q}}
			{\bigoplus_{i = 0}^{p-1} y^i \, \Omega_Y(B)_Q} \cong \bigoplus_{i = 0}^{p-1} 
			y^i \cdot \frac{t^{-j_{max}(i)} \Omega_{Y, Q}}{\Omega_Y(B)_Q}.
		\end{equation*}
		The statement follows easily by noting that the images of the forms $dt/t^j$
		for $j = 2, \ldots, j_{max}(i)$ form a basis of $\frac{t^{-j_{max}(i)} \Omega_{Y, Q}}{\Omega_Y(B)_Q}$.
		
	(2) Let $j_{min}(i)$ denote the least integer satisfying the inequality $pj - mi \ge 1$, i.e. $j_{min}(i) =
	\left\lceil \frac{mi+1}{p} \right\rceil$. Analogously, using \eqref{eqn:valuation_of_f_Zp} one obtains:
		\begin{equation*}
			H^1_Q \cong \bigoplus_{i = 0}^{p-1} y^i \frac{\mc O_Y(-B)_Q}{t^{j_{min}(i)} \mc O_{Y, Q}}.
		\end{equation*}
		To finish the proof of (2), it suffices to notice that the images of the elements
		$t^j$ for $j = 1, \ldots, j_{min}(i)-1$ form a basis of $\frac{\mc O_Y(-B)_Q}{t^{j_{min}(i)} \mc O_{Y, Q}}$.
		
	(3) In the sequel we will use several times the identity:
	\begin{equation} \label{eqn:jmin_jmax_identity}
		j_{max}(i) + j_{min}(i+1) = m+1 \qquad \textrm{ for every } 0 \le i \le p-2,
	\end{equation}
	which may be obtained from explicit formulas for $j_{max}(i)$ and $j_{min}(i)$ given above. Note that for any $h \in k(X)$ we have $\ord_P(dh) \ge \ord_P(h) - 1$.
	Using this fact for $h := y^p - y - \frac{1}{t^m}$ (which belongs to $\mf m_{X, P}^{2mp}$ by~\eqref{eqn:suitable_uniformizer}) we obtain $dy \equiv m \cdot dt/t^{m+1} \pmod{\mf m_{X, P}^{2mp - 1} \Omega_{X, P}}$,
		which easily implies that:
		\[
		y^i \cdot t^{m+1 - j} \, dy \equiv m \cdot y^i dt/t^j \pmod{\mf m_{X, P}^{(2mp - 1) + (- mi + (m+1-j)p)} \Omega_{X, P}}.
		\]
		Thus, since $(2mp - 1) + (- mi + (m+1-j)p) \ge -1$:
		\begin{align*}
			\frac{y^i \, dt}{t^j} - d \left( \frac{1}{(i+1) \cdot m} y^{i+1} t^{m+1 - j} \right)
			&\equiv - \frac{(m+1 - j)}{(i+1) \cdot m} y^{i+1} t^{m - j} \, dt\\
			&\equiv 0 \pmod{\Omega_X(R)_Q},
		\end{align*}
		since by~\eqref{eqn:valuation_of_diff_form}:
		\begin{align*}
			\ord_P(y^{i+1} t^{m - j} \, dt) &= -(i+1) \cdot m + (m-j) \cdot p
			+ (p-1) \cdot (m+1)\\
			&\ge -(p-1) \cdot m + 0 \cdot p
			+ (p-1) \cdot (m+1)\\
			&= p-1 \ge 1.
		\end{align*}
		Hence the listed elements define valid elements of $H^1_{dR, Q}$. Moreover:
		\begin{itemize}
			\item part (1) implies that the elements
			\[
			\left(\frac{y^i \, dt}{t^j}, \frac{1}{(i+1) \cdot m} y^{i+1} t^{m+1 - j} \right) =
			\left(\frac{y^i \, dt}{t^j}, 0 \right) \qquad \textrm{ in } H^1_{dR, Q}
			\]
			for $i = 0, \ldots, p-2$ and $j = 2, \ldots, j_{max}(i)$,
			are images of a basis of $H^0_Q$ through the map $H^0_Q \to H^1_{dR, Q}$ from Lemma~\ref{lem:properties_IXY_H1dRQ},
			\item[] (We used here that $y^{i+1} t^{m+1 - j} = 0$ in $H^1_Q$, since $m+1 - j \ge j_{min}(i+1)$ by~\eqref{eqn:jmin_jmax_identity}.)
			\item part (2) implies that the images of the elements:
			\[
			\left(\frac{y^i \, dt}{t^j}, \frac{1}{(i+1) \cdot m} y^{i+1} t^{m+1 - j} \right)
			\]
			for $i = 0, \ldots, p-2$ and $j = j_{max}(i) + 1, \ldots, m$ through the map $H^1_{dR, Q} \to H^1_Q$ form a basis
			of $H^1_Q$ (we used here~\eqref{eqn:jmin_jmax_identity} again).
		\end{itemize}
		Therefore the listed elements form a basis of $H^1_{dR, Q}$ by Lemma~\ref{lem:properties_IXY_H1dRQ}.
\end{proof}

\begin{proof}[Proof of Corollary~\ref{cor:cohomology_of_Zp}]
	For any $j$, let $i_{max}(j)$ denote the largest integer satisfying
	$mi + pj \le (m+1) \cdot (p-1) + 1$, i.e. $i_{max}(j) = \lfloor \frac{(m+1) \cdot (p-1) + 1 - pj}{m} \rfloor$. By Proposition~\ref{prop:H0Q_for_Zp}:
	\begin{align*}
		H^0_Q = \bigoplus_{j \ge 2} \frac{dt}{t^j} \cdot \Span_k(1, y, y^2, \ldots, y^{i_{max}(j)})
		\cong \bigoplus_{j \ge 2} J_{i_{max}(j) + 1}.
	\end{align*}
	For any $1 \le i \le p-1$:
	\begin{align*}
		\# \{j \ge 2 : i_{max}(j) + 1 = i \} &= \# \left\{j \ge 2 : 0 < i-\frac{(m+1) \cdot (p-1) + 1 - pj}{m} \le 1 \right\}\\
		&= \# \left\{j \ge 2 : (m+1) - \frac{m \cdot (i+1)}{p} < j \le (m+1) - \frac{m \cdot i}{p} \right\}\\
		&= \left\lceil \frac{m \cdot (i+1)}{p} \right\rceil - \left \lceil \frac{m \cdot i}{p} \right\rceil = \alpha_Q(i).
	\end{align*}
	Thus $H^0_Q = \bigoplus_{i = 1}^{p-1} J_i^{\alpha_Q(i)}$.
	This allows to find $H^1_Q$ using Lemma~\ref{lem:properties_H0Q_H1Q}~(1), since the representations $J_i$ are self-dual.
	We compute now $H^1_{dR, Q}$. Denote:
	\[
		\upsilon_{i, j} := \left(\frac{y^i \, dt}{t^j}, \frac{1}{(i+1) \cdot m} y^{i+1} t^{m+1 - j} \right) \in H^1_{dR, Q}
	\]
	for $0 \le i \le p-2$ and $2 \le j \le m$ (note that $\upsilon_{i, j}$ define a basis of $H^1_{dR, Q}$ by Proposition~\ref{prop:H0Q_for_Zp}~(3)). Note that for any $2 \le j \le m$ we have:
	\begin{align*}
		\tr_{X/Y}\left( \upsilon_{p-2, j} \right) = \left(0, \frac{1}{m} t^{m+1 - j} \right) =(0, 0),
	\end{align*}
	since $t^{m+1 - j} \in \mf m_{Y, Q}$. Hence, arguing as in and after the proof of Lemma~\ref{lem:Ji_and_Span_yi}, the $k[G]$-module generated by $\upsilon_{p-2, j}$ is $\Span_k(\upsilon_{i, j} : 0 \le i \le p-2) \cong J_{p-1}$.
	Thus, using Proposition~\ref{prop:H0Q_for_Zp}:
	\begin{align*}
		H^1_{dR, Q} = \bigoplus_{j = 2}^m \Span_k\left(g^*(\upsilon_{p-2, j}) : g \in G \right) \cong J_{p-1}^{\oplus (m - 1)}.
	\end{align*}
	The proof follows by Theorem~\ref{thm:main_thm}.
\end{proof}
\subsection{$\ZZ/p$-cover without a magical element} \label{subsec:no_magical_element}
The goal of this subsection is to construct a $\ZZ/p$-cover without a magical element
(and thus without a global standard form).
Let $p \ge 5$ be a prime and $k = \overline{\FF}_p$. Pick a general smooth projective curve $Y/k$ of genus $g_Y \ge p/2$.
Then its gonality equals $\gamma = \lfloor \frac{g_Y + 3}{2} \rfloor$ by Brill--Noether theory, see e.g.~\cite[Proposition A.1. (v)]{Poonen_Gonality_Modular}. 
Let $m$ and $M$ be integers satisfying $p \nmid m$ and:
\[
	\frac{2 g_Y}{p} < \frac{m}{p} < M < \gamma.
\]
One can take for example $M = \gamma - 1$ and $m = 2g_Y + 1$ or $m = 2g_Y + 2$ (depending on whether $p | 2g_Y + 1$).
Fix a point $Q \in Y(k)$ and let $t \in k(Y)$ be a uniformizer at $Q$. Using the fact that $H^1(Y, \mc O_Y(2 g_Y \cdot Q)) = 0$
(e.g. by Riemann--Roch theorem), we may choose a function $f \in k(Y)$ that is regular outside of~$Q$ and satisfies:
\[
	f \equiv \frac{1}{t^{Mp}} - \frac{1}{t^M} + \frac{1}{t^m} \pmod{\mf m_{Y, Q}^{-2g_Y}}.
\]
Indeed, by Lemma~\ref{lem:description_of_sheaf_cohomology} for $\mc F := \mc O_Y(2g_Y \cdot Q)$, $U := Y \setminus \{ Q \}$, $S = \{ Q \}$ and we have:
\[
	\coker \left( \mc O_Y(U) \to k(Y)/\mf m_{Y, Q}^{-2g_Y} \right) \cong H^1(Y, \mc O_Y(2g_Y \cdot Q)) = 0.
\]
In particular, there exists $f \in \mc O_Y(U)$, which maps to $\frac{1}{t^{Mp}} - \frac{1}{t^M} + \frac{1}{t^m}$ in $k(Y)/\mf m_{Y, Q}^{-2g_Y}$.\\

Let $\pi : X \to Y$ be the $\ZZ/p$-cover defined by the equation~\eqref{eqn:artin_schreier}. 
Suppose to the contrary that $z = \sum_{i = 0}^{p-1} z_i \cdot y^i \in k(X)$
is a magical element. By~\eqref{eqn:trace_of_dual_z} and~\eqref{eqn:AS_trace_of_yi} we may assume that $z_{p-1} = -\tr_{X/Y}(z) = 1$. Let $\pi^{-1}(Q) = \{ P \}$. Note that for $Q' \in Y(k)$, $Q' \neq Q$, we have $\ord_{Q'}(y) \ge 0$.
Therefore, by~\eqref{eqn:AS_trace_of_yi} and~\eqref{eqn:valuation_of_f_Zp}:
\[
	z_{p-2} = - \tr_{X/Y}(y \cdot z) + \tr_{X/Y}(y^p) \in \mc O_{Y, Q'}.
\]
Observe also that $y_1 := y - \frac{1}{t^M}$ is in local standard form at $Q$ and $m_Q = m$. Since $z = y_1^{p-1} + y_1^{p-2} \cdot (z_{p-2} - \frac{1}{t^M}) + \ldots$, we deduce:
\begin{align*}
	\ord_P \left(y_1^{p-2} \cdot (z_{p-2} - \frac{1}{t^M}) \right) &\ge
	 \ord_P(z) \ge -d_P' = -(p-1) \cdot m,
\end{align*}
where we used~\eqref{eqn:valuation_of_f_Zp} for the first inequality.
Thus
\[ \ord_Q \left(z_{p-2} - \frac{1}{t^M} \right) \ge -m/p, \]
and $z_{p-2} \in H^0(Y, \mc O_Y(M \cdot Q))$, $z_{p-2} \not \in k$.
On the other hand, by the definition of gonality, $H^0(Y, \mc O_Y(D)) \subset k$ for any divisor $D$ of degree less then $\gamma$. This yields
a contradiction with $M < \gamma$.

\section{Constructing magical elements} \label{sec:constructing_magical_elements}
In this section we show that a generic $p$-group cover has a magical element.
We prove this using the global standard form. The following observation shows that everything comes down to
constructing magical elements for $\ZZ/p$-covers.
\begin{Lemma} \label{lem:new_magical_elts}
	Suppose that $G$ is a finite $p$-group and $k$ is an algebraically closed field of characteristic $p$.
	Let $\pi : X \to Y$ be a $G$-cover of smooth projective curves over~$k$. Suppose that $\pi$ factors
	through a Galois cover $X' \to Y$. If $z_1 \in k(X)$ is a magical element for the cover $X \to X'$ and $z_2 \in k(X')$ 
	is a magical element for $X' \to Y$ then $z_1 \cdot z_2$ is a magical element
	for $X \to Y$.
\end{Lemma}
\begin{proof}
	Let $P \in \pi^{-1}(Q)$ and let $P'$ be the image of $P$ on $X'$. Then:
	\begin{align*}
		\ord_P(z_1 \cdot z_2) &= \ord_P(z_1) + e_{X/X', P} \cdot \ord_{P'}(z_2)\\
		&\ge -d_{X/X', P}' - e_{X/X', P} \cdot d_{X'/Y, P'}' = -d_P'.
	\end{align*}	
	(we used the transitivity of the different for the last equality, see e.g. \cite[\S III.4, Proposition 8]{Serre1979}).	Moreover, since $\tr_{X/X'}(z_1), \tr_{X'/Y}(z_2) \in k^{\times}$ by~\eqref{eqn:trace_of_dual_z}:
	\begin{align*}
		\tr_{X/Y}(z_1 \cdot z_2) = \tr_{X'/Y}( \tr_{X/X'}(z_1 \cdot z_2))
		= \tr_{X/X'}(z_1) \cdot \tr_{X'/Y}(z_2) \neq 0.
	\end{align*}
	This ends the proof.
\end{proof}
Let $k$, $G$, $\pi : X \to Y$ be as in Lemma~\ref{lem:new_magical_elts}. We say that $\pi$ has a \bb{global standard form}, if there exist $y_1, \ldots, y_n \in k(X)$ and a composition series
\begin{equation*} \label{eqn:composition_series}
0 = G_0 \unlhd G_1 \unlhd \ldots \unlhd G_n = G
\end{equation*}
such that $y_i$ is the Artin Schreier generator in global standard form of the $\ZZ/p$-cover $X/G_{i - 1} \to X/G_i$
for $i = 1, \ldots, n$. Lemmas~\ref{lem:gsf_gives_me_for_Zp} and~\ref{lem:new_magical_elts} imply that in this case $z := y_1^{p-1} \cdot \ldots \cdot y_n^{p-1} \in k(X)$ is a magical element for $\pi$. The notion of a global standard form of a $p$-group cover appeared already in~\cite{WardMarques_HoloDiffs}, where it was used to construct a basis of holomorphic differentials of $X$ in some cases.

\subsection{Generic $p$-group covers}
The goal of this subsection is to prove Theorem~\ref{thm:generic_intro}. Before the proof we review briefly the theory of moduli of $p$-group covers from~\cite{Harbater_moduli_of_p_covers}.\\

Let $U$ be a non-empty affine open of a smooth
projective curve $Y$ over $k$ and fix $u \in U(k)$. Denote $B := Y(k) \setminus U(k)$. Let $M_{U, G}$ be the moduli space of
pointed \'{e}tale $G$-covers of $(U, u)$, as defined by Harbater in~\cite{Harbater_moduli_of_p_covers}. Note that such covers correspond bijectively to pointed covers of $(Y, u)$ unramified over $U$.\\

In order to describe the structure of $M_{U, G}$, one proceeds inductively.
Let $H \cong \ZZ/p$ be a central normal subgroup of $G$ and let $G' := G/H$. Consider the map:
\[
	\pr : M_{U, G} \to M_{U, G'}, \quad (X \to Y) \mapsto (X/H \to Y).
\]
One can show that $M_{U, G}$ is a principal $M_{U, H}$-homogenous space over $M_{U, G'}$ (see e.g. proof of \cite[Theorem 1.2]{Harbater_moduli_of_p_covers}). The structure of a principal homogenous space is given by the map:
\[
M_{U, G} \times M_{U, H} \to M_{U, G}, \quad (X, Z) \mapsto X_Z,
\]
where $X : y_1^p - y_1 = f \in k(X')$, $Z : y_2^p - y_2 = g \in k(Y)$, $X_Z : y_3^p - y_3 = f + g$. Moreover, one proves that the map $\pr$ has
a (non-canonical) section and thus $M_{U, G} \cong M_{U, G'} \times M_{U, H}$. 
This allows one to show that $M_{U, G}$ is a direct limit of affine spaces.\\

The following lemma shows that (under some mild assumptions) the cover
$X_Z \to X'$ is at least as ramified as the cover $Z \to Y$.
\begin{Lemma} \label{lem:mQ_of_XZ}
	Let $X'$, $X$, $Z$ and $X_Z$ be as above.  Let also $P \in X'(k)$ and denote by $Q$ its image on $Y$.
	If $m_{Z/Y, Q} > m_{X/X', P} + d_{X'/Y, P}'$ then
	$m_{X_Z/X', P} > m_{Z/Y, Q}$. 
\end{Lemma}
\begin{proof}
	Assume that $y_1$ and $y_2$ are in local standard form at $P$ and $Q$ respectively.
	Note that (since $p \nmid \ord_Q(g)$):
	\begin{align*}
		\ord_P(dg/g) &= e_{X'/Y, P} \cdot \ord_Q(dg/g) + d_{X'/Y, P}\\
		&= - e_{X'/Y, P} + d_{X'/Y, P}.
	\end{align*}
    Therefore:
	\begin{equation} \label{eqn:ord_dg}
		\ord_P(dg) = e_{X'/Y, P} \cdot \ord_Q(g) + d_{X'/Y, P}' - 1.
	\end{equation}
	This yields:
	\[
		\ord_{P}(dg) = - e_{X'/Y, P} \cdot m_{Z/Y, Q} + d_{X'/Y, P}' - 1 < -m_{X/X', P} - 1 = \ord_{P}(df)
	\]
	and $\ord_{P}(df + dg) = \ord_{P}(dg)$. One easily checks that~\eqref{eqn:assumption_to_compute_mQ} holds for
	$g$ at $P$. Hence the result follows by Lemma~\ref{lem:computing_mQ}.
\end{proof}
In order to prove Theorem~\ref{thm:generic_intro} we need the following two topological lemmas. 
\begin{Lemma} \label{lem:topological}
	Let $f : \mc X \to \mc Y$ be an open continuous map between topological spaces. Suppose that $\mc U \subset \mc X$.
	If there exists a dense subset $\mc V \subset \mc Y$ such that
	$f^{-1}(y) \cap \mc U$ is dense in $f^{-1}(y)$ for every $y \in \mc V$, then $\mc U$ is dense in $\mc X$.
\end{Lemma}
\begin{proof}
	Let $\mc Z$ be a non-empty open subset of $\mc X$, we will show that $\mc Z \cap \mc U \neq \varnothing$.
	Note that $f(\mc Z)$ is a non-empty open subset of $\mc Y$ and hence $f(\mc Z) \cap \mc V \neq \varnothing$.
	Let $y \in f(\mc Z) \cap \mc V$. Then $\mc Z \cap f^{-1}(y)$ is a non-empty open subset
	of $f^{-1}(y)$. By assumption, $f^{-1}(y) \cap \mc U$ is dense in $f^{-1}(y)$.
	Hence:
	\[
	(f^{-1}(y) \cap \mc Z) \cap (f^{-1}(y) \cap \mc U) \neq \varnothing,
	\]
	which proves the lemma.
\end{proof}
\begin{Lemma} \label{lem:open_maps_direct_limit}
	Let $A := \varinjlim_{i \in I} A_i$ and $B := \varinjlim_{i \in I} B_i$ be direct limits of topological spaces over a directed set $I$. Suppose that $f : A \to B$ is a direct limit of continuous open maps $f_i : A_i \to B_i$. Then $f$ is an open map.
\end{Lemma}
\begin{proof}
	Let $\alpha_{ij} : A_i \to A_j$, $\beta_{ij} : B_i \to B_j$ be the transition maps of the corresponding direct systems. Denote also by $\alpha_i : A_i \to A$ and $\beta_i : B_i \to B$ the natural maps to the direct limits.
	Let $U \subset A$ be an open subset. In order to show that $f(U)$ is open,
	we need to prove that $\beta_i^{-1}(f(U))$ is open in $B_i$ for every $i \in I$. This follows from the identity:
	\[
	\beta_i^{-1}(f(U)) = \bigcup_{j > i} \beta_{ij}^{-1}(f_j(U_j)),
	\]
	where $U_j := \alpha_j^{-1}(U)$ for every $j \in I$.
\end{proof}

\begin{proof}[Proof of Theorem~\ref{thm:generic_intro}]
	Recall that any $G$-cover with global standard form satisfies~\ref{enum:B} (see the discussion at the beginning of Section~\ref{sec:constructing_magical_elements}). Thus it suffices to prove that the set of $G$-covers of $Y$ unramified over~$U$ with global standard form that satisfy~\ref{enum:A} (which we denote by $S_{U, G}$)
	is dense in $M_{U, G}$. We show this by induction on $\# G$. For $\# G = 1$ this is trivial. Let $G'$ and $H$ be as above.
	By induction hypothesis, the set $S_{U, G'}$ is dense in $M_{U, G'}$.
	By Lemma~\ref{lem:criterion_for_gsf}
	for every $X' \in S_{U, G'}$ the covers from the set
	\[
	\mc A := \{ X \in \pr^{-1}(X') : m_{X/X', P'} > 2g_{X'} \cdot p \qquad \forall_{P' \in B' := \textrm{ preimage of } B \textrm{ in } X'(k)} \} 
	\]
	have global standard form. Moreover, for every $X \in \mc A$
	and $P \in X(k)$, we have $H_P = H$. Hence for every $P \in X'(k)$, $G_P$ is the preimage of $G_{P'}'$ by the map $G \to G'$ (where $P' \in X'(k)$ is the image of $P$). This implies that $G_P$ is a normal subgroup of $G$ for every $P \in X(k)$. Thus $\mc A \subset \pr^{-1}(X') \cap S_{U, G}$.\\
	On the other hand, by Lemma~\ref{lem:mQ_of_XZ} the set $\mc A$ contains the set:
	\[
	\mc B := \{ X_Z : Z \in M_{U, H}, \, \, m_{Z/Y, Q} > N_0 \quad \forall_{Q \in B} \}
	\]
	for any fixed $X \in \pr^{-1}(X')$ and sufficiently large $N_0$. But the set $\{ Z \in M_{U, H} : m_{Z/Y, Q} > N_0 \quad \forall_{Q \in B} \}$ is dense in
	$M_{U, H}$. Indeed, the considered subset is contained in a complement of a finite dimensional subspace in $M_{U, H}$. 
	Finally, note that the map $\pr : M_{U, G} \to M_{U, G'}$ is open. Indeed, it might be identified with the limit of projections $\varinjlim_i \AA^{n_i + m_i} \to \varinjlim_i \AA^{n_i}$
	and all of them are open by \cite[tag 037G]{stacks-project}.
	Thus the proof follows by Lemmas~\ref{lem:topological} and~\ref{lem:open_maps_direct_limit}.
\end{proof}

\subsection{Example: Heisenberg group covers} \label{subsec:example_heisenberg}
Consider the Heisenberg group $\textrm{mod } p$:
\[
E(p^3) = \langle a, b, c : a^p = b^p = c^p = e, c = [a, b], c \in Z(E(p^3)) \rangle.
\]
Let $Y$ be a smooth projective curve over $k$ and let $f_1, f_2, f_3 \in k(Y)$.
Define $X$ to be a smooth projective curve with the function field $k(Y)(y_1, y_2, y_3)$, where:
\begin{equation} \label{eqn:Ep3cover}
\begin{aligned} 
	y_1^p - y_1 &= f_1,\\
	y_2^p - y_2 &= f_2,\\
	y_3^p - y_3 &= f_3 + f_2 \cdot (y_1 - y_2).
\end{aligned}	
\end{equation}
Then $E(p^3)$ acts on $X$ via the formulas:
\begin{align*}
	a^*(y_1, y_2, y_3) &= (y_1 + 1, y_2, y_3 + y_2)\\
	b^*(y_1, y_2, y_3) &= (y_1 + 1, y_2 + 1, y_3)\\
	c^*(y_1, y_2, y_3) &= (y_1, y_2, y_3 - 1).
\end{align*}
It turns out that every $E(p^3)$-cover of $Y$ is of this form (this follows from results of Saltman, cf.~\cite{Saltman_Noncrossed_product}).\\

We use now this description of $E(p^3)$-covers to give an explicit example of covers
satisfying~\ref{enum:A} and~\ref{enum:B}.
Suppose that $p > 2$ and $Y$ is the elliptic curve:
\begin{equation*} 
	Y : w^2 = (x - \alpha_1) \cdot (x - \alpha_2) \cdot (x - \alpha_3)
\end{equation*}
over $k$, where $\alpha_1, \alpha_2, \alpha_3 \in k$ are any distinct elements. Pick any natural numbers $a_1, a_2, b_2, a_3, b_3, c_3$, which are non-divisible by $p$ and satisfy:
\[
a_1 > p, \quad a_2 > 4a_1 \cdot p, \quad a_3 > 4 \cdot p \cdot (a_1 + a_2 + b_2), \quad b_3 > 4 b_2.
\]
Define $X$ to be the $E(p^3)$-cover of $Y$ given by~\eqref{eqn:Ep3cover}, where:
\begin{align*}
	f_1 &= (x - \alpha_1)^{-a_1},\\
	f_2 &= (x - \alpha_1)^{-a_2} \cdot (x - \alpha_2)^{-b_2},\\
	f_3 &= (x - \alpha_1)^{-a_3} \cdot (x - \alpha_2)^{-b_3} \cdot (x - \alpha_3)^{-c_3}.
\end{align*}
\begin{Corollary} \label{cor:example_from_intro}
	The $E(p^3)$-cover $X \to Y$ given as above satisfies the conditions~\ref{enum:A} and \ref{enum:B}. Moreover:
	\[
	G_{Q_1} \cong E(p^3), \quad G_{Q_2} \cong \ZZ/p \times \ZZ/p, \quad G_{Q_3} \cong \ZZ/p,
	\]
	where $Q_i := (\alpha_i, 0) \in Y(k)$.
\end{Corollary}
\begin{proof}
	In the proof we will estimate genus of a curve in a tower of covers in the following way. If $Z \to Y$ is
	an Artin--Schreier cover with the equation~\eqref{eqn:artin_schreier} then Riemann--Hurwitz formula yields:
	\begin{equation} \label{eqn:genus_AS}
		g_Z < p \cdot (g_Y + \deg_Y f),
	\end{equation}
	where $\deg_Y f$ is the degree of $f$ treated as a morphism $Y \to \PP^1$.\\
	
	Let $X_1$ be the $\ZZ/p$-cover of $Y$, given by $y_1^p - y_1 = f_1$
	and let $X_2$ be the $\ZZ/p \times \ZZ/p$-cover of $Y$, given by:
	\[
	y_1^p - y_1 = f_1, \quad y_2^p - y_2 = f_2.
	\]
	Note that by~\eqref{eqn:mQ_formula}:
	\[
	m_{X_1/Y, Q_1} = 2a_1 > 2g_Y \cdot p
	\]
	and hence $X_1 \to Y$ has a global standard form by Lemma~\ref{lem:criterion_for_gsf}.
	By~\eqref{eqn:genus_AS}, $g_{X_1} < p \cdot (1 + 2a_1) < 3a_1 \cdot p$.
	One checks that $f_2$ satisfies the condition~\eqref{eqn:assumption_to_compute_mQ}.
	Hence, using Lemma~\ref{lem:computing_mQ} and~\eqref{eqn:ord_dg},
	for $P \in X_1(k)$ in the preimage of $Q_1$:
	\begin{align*}
		m_{X_2/X_1, P} &= -\ord_P(df_2) - 1\\
		&= p \cdot 2a_2 - d_{X_1/Y, P}' - 1\\
		&= p \cdot 2a_2 - (p-1) \cdot 2 a_1 - 1\\
		&> 2g_{X_1} \cdot p. 
	\end{align*}
	Hence, by Lemma~\ref{lem:criterion_for_gsf}, $X_2 \to X_1$ also has a global standard form.
	Note now that by~\eqref{eqn:genus_AS}:
	\begin{align*}
		g_{X_2} &< p \cdot (g_{X_1} + \deg_{X_1} f_2) < 3p^2 \cdot (a_1 + a_2 + b_2).
	\end{align*}
	Using the transitivity of the different (see e.g. \cite[\S III.4, Proposition 8]{Serre1979}), the equality $m_{X_1/Y, Q_1} = 2a_1$ and inequality $m_{X_2/X_1, P} < p \cdot 2a_2$ for $P \in X_1(k)$ in the preimage of $Q_1$, we obtain:
	\begin{align*}
		d'_{X_2/Y, Q_1} = p \cdot d'_{X_1/Y, Q_1} + d'_{X/X_1, Q_1} &< 2p^2 \cdot (a_1 + a_2).
	\end{align*}
	Therefore, by~\eqref{eqn:ord_dg}, for any $P \in X_2(k)$ in the preimage of $Q_1$:
	\begin{align*}
		-\ord_P(df_3) = -p^2 \cdot \ord_{Q_1}(f_3) - d'_{X_2/Y, Q_1} + 1 &> p^2 \cdot 2a_3 - 2p^2 \cdot (a_1 + a_2).
	\end{align*}
	One easily checks that~\eqref{eqn:assumption_to_compute_mQ} is satisfied for
	$f_3 + (y_1 - y_2) \cdot f_2$ and that $\ord_P(d(f_3 + (y_1 - y_2) \cdot f_2)) = \ord_P(df_3)$.
	Hence:
	\begin{align*}
		m_{X/X_2, P} = -\ord_P(d(f_3 + (y_1 - y_2) \cdot f_2)) - 1
		> 2g_{X_2} \cdot p.
	\end{align*}
	Therefore $X \to Y$ has a global standard form. Also, $G_{Q_1} = E(p^3)$, which implies that $X$ is an irreducible curve (recall that any connected Artin-Schreier cover is irreducible). Moreover, one checks that $m_{X/X_2, P} > 0$ for $P \in X(k)$ in the preimage of $Q_2$. Hence $G_{P} = \ZZ/p \times \ZZ/p$ (since $X \to X_1$ is totally ramified over $P$ and $X_1 \to Y$ is unramified
	over $Q_2$). Analogously, $G_{Q_3} = \ZZ/p$. This finishes the proof. 
\end{proof}

\bibliography{bibliografia}
\end{document}